\documentclass[10pt,a4paper,reqno]{amsart}

\usepackage[english]{babel}
\usepackage{amssymb,amsmath,amsthm}
\usepackage{cancel}
\usepackage{graphicx}
\usepackage{xcolor}
\usepackage{soul}
\usepackage{mathrsfs}
\usepackage{enumitem}
\usepackage[immediate,debrief]{silence}
\usepackage[pagebackref]{hyperref}
\usepackage{color}

\usepackage{scrtime}

\renewcommand\mathcal{\mathscr}
\theoremstyle{plain}
\newtheorem{theorem}{Theorem}[section]
\newtheorem*{theorem*}{Theorem}
\newtheorem{lemma}[theorem]{Lemma}

\newtheorem{proposition}[theorem]{Proposition}
\theoremstyle{remark}
\newtheorem{remark}[theorem]{Remark}
\numberwithin{equation}{section}
\theoremstyle{definition}
\newtheorem{definition}[theorem]{Definition}
\numberwithin{equation}{section}
\theoremstyle{notation}

\numberwithin{equation}{section}
\theoremstyle{Basic assumptions}
\newtheorem{Basic assumptions}[theorem]{Basic assumptions}
\numberwithin{equation}{section}

\newcount\quantno
\everydisplay{\quantno=0}\everycr{\quantno=0}
\newcommand\quant{\advance\quantno by1
                      \ifnum\quantno=1\qquad\else\quad\fi\forall }
\newcommand\itemno[1]{(\romannumeral #1)}

\renewcommand\Re{\operatorname{\mathrm{Re}}}
\renewcommand\Im{\operatorname{\mathrm{Im}}}

\newcommand\rest[1]{\kern-.1em
          \lower.5ex\hbox{$\scriptstyle #1$}\kern.05em}

\newcommand\diam[1]{\mathrm{diam}#1}
\newcommand\set[1]{{\left\{#1\right\}}}

\renewcommand\mod[1]{\vert{#1}\vert}
\newcommand\bigmod[1]{\bigl\vert{#1}\bigr|}
\newcommand\Bigmod[1]{\Bigl\vert{#1}\Bigr|}

\newcommand\norm[2]{{\Vert{#1}\Vert_{#2}}}

\newcommand\bignorm[2]{\big\Vert{#1}\big\Vert_{#2}}
\newcommand\bignormto[3]{\big\Vert{#1}\big\Vert_{#2}^{#3}}

\newcommand\bigopnormto[3]{\big|\!\big|\!\big| {#1} \big|\!\big|\!\big|_{#2}^{#3}}
\newcommand\bigopnorm[2]{\bigl|\!\bigl|\!\bigl| {#1}
\bigr|\!\bigr|\!\bigr|_{#2}}

\newcommand\wrt{\,\text{\rm d}}

\newcommand\bD{\mathbf{D}}

\newcommand\BB{\mathbb{B}}

\newcommand\BN{\mathbb{N}}
\newcommand\BR{\mathbb{R}} 

\newcommand\BS{\mathbb{S}}

\newcommand\BX{\mathbb{X}}

\newcommand\cA{\mathcal{A}}  

\newcommand\cB{\mathcal{B}}

\newcommand\cE{\mathcal{E}}

\newcommand\cI{\mathcal{I}}

\newcommand\cM{\mathcal{M}}   
\newcommand\cN{\mathcal{N}}

  \newcommand\fX{\mathfrak{X}}

\newcommand\al{\alpha}
\newcommand\be{\beta}
\newcommand\ga{\gamma}    \newcommand\Ga{\Gamma}
\newcommand\de{\delta}
  \newcommand\vep{\varepsilon}

\newcommand\la{\lambda}   \newcommand\La{\Lambda}
\newcommand\om{\omega}    \newcommand\Om{\Omega}  
\newcommand\si{\sigma}
\newcommand\te{\theta}
\newcommand\vp{\varphi}
\newcommand\vr{\varrho}

\newcommand\OV{\overline}
\newcommand\funnyk{k\hbox to 0pt{\hss\phantom{g}}}

\newcommand\lu[1]{L^1(#1)}
\newcommand\lp[1]{L^p(#1)}
\newcommand\laq[1]{L^q(#1)}

\newcommand\ly[1]{L^\infty(#1)}
\newcommand\lorentz[3]{L^{#1,#2}(#3)}

\newcommand\wt{\widetilde}
\newcommand\whH{\widehat{\phantom{G}}\hbox to 0pt{\hss $H$}}

\newcommand\emspace{\hbox to 6pt{\hss}}
\newcommand\ds{\displaystyle}

\newcommand\rmi{\hbox{\rm (i)}}
\newcommand\rmii{\hbox{\rm (ii)}}
\newcommand\rmiii{\hbox{\rm (iii)}}
\newcommand\rmiv{\hbox{\rm (iv)}}

\newcommand\ioty{\int_0^{\infty}}
\newcommand\dtt[1]{\,\frac{\mathrm {d} #1}{ #1}}

\newcommand\One{{\mathbf{1}}}

\newcommand\e{\mathrm{e}}

\newcommand\arcsinh{\mathrm{arcsinh}}

\newcommand\cg[1] {{\color{red}#1}}

\DeclareSymbolFont{EUEX}{U}{euex}{m}{n}

\DeclareSymbolFont{euexlargesymbols}{U}{euex}{m}{n}
\DeclareMathSymbol{\intop}{\mathop}{euexlargesymbols}{"52}
     \def\int{\intop\nolimits}

\DeclareSymbolFont{euexsymbols}     {U}{euex}{m}{n}
\DeclareMathSymbol{\smallint}{\mathop}{euexsymbols}{"52}

\newcommand{\Hmm}[1]{\leavevmode{\marginpar{\tiny%
			$\hbox to 0mm{\hspace*{-0.5mm}$\leftarrow$\hss}%
			\vcenter{\vrule depth 0.1mm height 0.1mm width \the\marginparwidth}%
			\hbox to 0mm{\hss$\rightarrow$\hspace*{-0.5mm}}$\\\relax\raggedright #1}}}

\begin{document}

\title[Hardy--Littlewood maximal operator]
{Hardy--Littlewood maximal operators \\
on certain manifolds with bounded geometry
}



\author[S. Meda, S. Pigola, A.G. Setti and G. Veronelli]
{Stefano Meda, Stefano Pigola, \\ Alberto G. Setti and Giona Veronelli}

\address{Stefano Meda:
Dipartimento di Matematica e Applicazioni
\\ Universit\`a di Milano-Bicocca\\
via R.~Cozzi 53\\ I-20125 Milano\\ Italy
\hfill\break
stefano.meda@unimib.it}

\address{Stefano Pigola:
Dipartimento di Matematica e Applicazioni
\\ Universit\`a di Milano-Bicocca\\
via R.~Cozzi 53\\ I-20125 Milano\\ Italy
\hfill\break
stefano.pigola@unimib.it}

\address{Alberto G. Setti:
DiSAT - Sezione di Matematica\\
Universita' dell'Insubria\\
via Valleggio 11\\ I-22100 Como\\  Italy
\hfill\break
alberto.setti@uninsubria.it}

\address{Giona Veronelli:
Dipartimento di Matematica e Applicazioni
\\ Universit\`a di Milano-Bicocca\\
via R.~Cozzi 53\\ I-20125 Milano\\ Italy
\hfill\break
giona.veronelli@unimib.it}

\begin{abstract}
	In this paper we study the $L^p$ boundedness of the centred and the uncentred Hardy--Littlewood maximal operators on certain
	Riemannian manifolds with bounded geometry. Our results complement those of various authors.  We show that,
	under mild assumptions, $L^p$ estimates for the centred operator are ``stable'' under conformal changes of the metric, and prove sharp~$L^p$ estimates for the 
	centred operator on Riemannian models with pinched negative scalar curvature.
	Furthermore, we prove that the centred operator is of weak type $(1,1)$ on the connected sum of two space forms with negative curvature, whereas
	the uncentred operator is, perhaps surprisingly, bounded only on $L^\infty$.  

	We also prove that if two locally doubling geodesic metric measure spaces enjoying the uniform ball size condition are strictly quasi-isometric,
	then they share the same boundedness properties for both the centred and the uncentred maximal operator.

	Finally, we discuss some $L^p$ mapping properties for the centred operator on a specific Riemannian surface introduced by Str\"omberg,
	providing new interesting results. 
\end{abstract}

\maketitle

\setcounter{section}{0}

\section{Introduction}

The purpose of this paper is to prove $L^p$ bounds for the centred and the uncentred Hardy--Littlewood (HL) maximal operators on
a class of Riemannian manifolds with \textit{bounded geometry}, in a sense we shall make precise later,  and exponential volume growth.

\smallskip
Suppose that $(X,d,\mu)$ is a metric measure space.
For the sake of notational convenience for any measurable subset $E$ of~$M$ we often write $\mod{E}$ instead of $\mu(E)$.
Denote by $B_r(x)$ the open ball with centre $x$ and radius $r$, i.e.  $B_r(x) := \{y\in X: d(x,y) < r\}$.
For each locally integrable function $f$ on $X$, we consider its \textit{centred} and its \textit{uncentred} HL maximal functions, defined by
$$
\cM f(x)
:= \sup_{r>0} \, \frac{1}{\mod{B_r(x)}} \, \int_{B_r(x)} \mod{f} \wrt \mu
\quad\hbox{\textrm{and}}\quad
\cN f(x)
:= \sup_{B \ni x} \, \frac{1}{\mod{B}} \, \int_{B} \, \mod{f} \wrt \mu,
$$
respectively, where the last supremum is taken over all open metric balls that contain $x$.

Obviously both $\cM$ and $\cN$ are bounded on $\ly{X}$.  It is natural to speculate whether either $\cM$ or $\cN$ is
bounded on $\lp{X}$ for some finite~$p$.  By the Marcinkiewicz interpolation theorem, the range of all $p$'s such that $\cM$ is bounded
on $\lp{X}$ is an interval, which we denote by $I_X$.  Similarly, $J_X$ will denote the interval of all $p$'s such that $\cN$ is bounded
on $\lp{X}$.

Consider also the local versions of $\cM$ and $\cN$, defined by
$$
\begin{aligned}
\cM_0 f(x)
	 := \sup_{0<r\leq 1} \, \frac{1}{\bigmod{B_r(x)}} \, \int_{B_r(x)} \, \mod{f} \wrt \mu
	 \qquad
\cN_0 f(x)
	 := \sup_{B \ni x: r_B\leq 1} \, \frac{1}{\bigmod{B}} \, \int_{B} \, \mod{f} \wrt \mu,
\end{aligned}
$$
where $r_B$ denote the radius of $B$.  Clearly $\cM_0$ and $\cN_0$ are bounded on $\ly{X}$.

It is worth mentioning that if $(X,d,\mu)$ is locally doubling (see \eqref{f: LDC} for the precise definition), then $I_X$ and $J_X$ depend only on the
``coarse geometry" of $X$, in the following sense.  A localisation argument, which hinges on a partition of unity associated to a $1$-discretisation of $X$, reduces
the problem of establishing a weak type $(1,1)$ estimate for $\cM_0$ and $\cN_0$ to a purely local matter, to which a slight variant of a classical
covering argument applies (see, for instance, \cite[Lemma, p.~9]{St}).  The Marcinkiewicz interpolation theorem then implies that
$\cM_0$ and $\cN_0$ are bounded on $\lp{X}$ for every $p$ in $(1,\infty]$.  Thus, $\cM$ [resp. $\cN$] is bounded on $\lp{X}$ for some $p$ in $(1,\infty)$
if and only if $\cM_\infty$ [resp. $\cN_\infty$] does, where
$$
\begin{aligned}
\cM_\infty f(x)
	& := \sup_{r> 1} \, \frac{1}{\bigmod{B_r(x)}} \, \int_{B_r(x)} \, \mod{f} \wrt \mu \\
\cN_\infty f(x)
	& := \sup_{B \ni x: r_B> 1} \, \frac{1}{\bigmod{B}} \, \int_{B} \, \mod{f} \wrt \mu.
\end{aligned}
$$
It is well known that if $(X,d,\mu)$ is (globally) doubling, then both $\cM$ and $\cN$ are of weak type $(1,1)$, whence $I_X$ and $J_X$ contain $(1,\infty]$.

In this paper we shall mainly be concerned with locally (but not globally) doubling metric measure spaces naturally associated to Riemannian manifolds.  Thus,
we shall consider an $m$-dimensional connected noncompact complete Riemannian manifold $M$, with Riemannian tensor $g$,
induced Riemannian distance~$d$, and Riemannian density $\mu$.  In particular, if $M$ has Ricci curvature bounded from below, then, by the 
relative volume estimate (see, e.g., \cite[Prop. 4.1]{CGT}), the Riemannian measure is locally doubling, and the considerations above apply.

The following two results of J.O.~Str\"omberg and A.D.~Ionescu on symmetric spaces $\BX$ of the noncompact type are landmarks of our investigation:
\begin{enumerate}
	\item[\itemno1]
		$\cM_\infty$ is of weak type $(1,1)$ \cite{Str}, whence $I_\BX = (1,\infty]$;
	\item[\itemno2]
		$\cN_\infty$ is bounded on $\lp{\BX}$ for every $p> 2$ \cite{I1} and it is also of restricted weak type $(2,2)$
		in the case where $\BX$ has real rank one~\cite{I2}.  Furthermore, if $p<2$, then $\cN_\infty$ is unbounded on $\lp{\BX}$.
		Thus, $J_\BX = (2,\infty]$ for noncompact symmetric spaces.
\end{enumerate}
J.-Ph.~Anker, E.~Damek and C.~Yacoub \cite[Corollary~3.22]{ADY} extended the work of Str\"omberg to Damek--Ricci spaces.  

It is worth pointing out that the proofs of the abovementioned results hinge on the very rich structure of noncompact symmetric spaces, and
it is seems difficult to adapt them to a larger class of manifolds with ``similar" geometric features but less ``rigid" structure.
Nevertheless, there are some generalisations of these results to Riemannian manifolds with nonconstant curvature,
due to Str\"omberg himself, N.~Lohou\'e and H.-Q.~Li.

Str\"omberg \cite[p.~126]{Str} proved that if $0<a<b<2a$ and $M$ is a surface isometric to the unit disc endowed with the metric
$\wrt s^2 = \la(z) \, \mod{\wrt z}^2$, with $\ds\lim_{\mod{z}\to 1^-} \, \la(z) = \infty$ and Gaussian curvature $\kappa$ satisfying the inequality
$$
-b^2\leq \kappa \leq -a^2,
$$
then~$I_M$ contains $(b/a,\infty]$.  He complemented this result by exhibiting a surface $\Pi$ with Gaussian curvature
satisfying the bound above for which~$\cM$ is unbounded on $\lp{\Pi}$ if $p<b/a$, thus proving that $I_{\Pi}$ is contained in $[b/a,\infty]$.
However, he gave only a suggestion for the proof of the unboundedness result: the reader will find full details in the proof of
Theorem~\ref{t: Str}~\rmi\ below.

We emphasise that Str\"omberg reduces the proof of the positive result to $L^p$ bounds for a certain
convolution operator on the hyperbolic half-plane.  The kernel of such operator belongs to $\lp{\Pi}$ for all $p>b/a$.  Thus, if $b<2a$,
then one can appeal to the Kunze--Stein phenomenon, and obtain the desired conclusion.  In fact, it is not hard to show that such kernel belongs
to weak-$L^{b/a}$, so that the sharp form of the Kunze--Stein phenomenon \cite{Co} implies that $\cM$ is also of restricted weak type $({b/a},{b/a})$.

It is natural to speculate what happens in the case where $b>2a$.  In Theorem~\ref{t: Str}~\rmii\ below we shall prove that in this case $\cM$ is bounded
on $\lp{\Pi}$ (if and) only if $p$ is equal to $\infty$.

Lohou\'e \cite{Lo} proved that if $M$ is a Cartan--Hadamard manifold with sectional curvature $K_\si$ satisfying the bound
\begin{equation} \label{f: bounds sect curv}
	-b^2 \leq K_\si \leq -a^2,
\end{equation}
and $0<a<b<5b/4$, then $\cM$ is bounded on $\lp{M}$ for every $p>a/(3a-2b)$.  This threshold index is an outcome of the method employed by the author,
and it is not sharp.

{The previous results lead to conjecture that if $M$ is a Cartan--Hadamard manifold with sectional curvature $K_\si$ 
satisfying the bound \eqref{f: bounds sect curv}, and $1 <a < b < 2a$, then $\cM$ is bounded on $\lp{M}$ for all $p>b/a$.  We are not able to
prove this conjecture; however, we show that if $M$ is a complete, noncompact $m$-dimensional \textit{model manifold} with pinched negative
\textit{scalar curvature}, i.e.
\begin{equation} \label{f: bounds scalar curv}
	-m(m-1)\, b^2 \leq \mathrm{Scal}_\si \leq -m(m-1)\,a^2,
\end{equation}
then $\cM$ is bounded on $\lp{M}$ for $p>b/a$ (see Theorem~\ref{t: scalar} below).  We obtain this result as a consequence of Theorem~\ref{t: conformal},
where we prove, under reasonably general assumptions, a ``conditional'' $L^p$ boundedness result for $\cM$ on manifold with conformally equivalent metrics.  
}

Interesting results on a class of conic manifolds with nonconstant curvature related to the geometric framework considered by Str\"omberg
can be found in \cite{L1,L2,L3} and in the papers cited therein.
In particular, \cite{L1} and \cite{L2} focus on an interesting class of conic manifolds~$X$ of the form $X_0\times (0,\infty)$, where $X_0$ is a
length metric measure space.  The distance~$d$ on~$X$ is defined implicitly by
$$
\cosh d\big((x_1,y_1), (x_2,y_2)\big)
:= (2y_1y_2)^{-1} \, \big[y_1^2+y_2^2+d_{X_0}(x_1,x_2)^2\big],
$$
where $x_1$ and $x_2$ are in $X_0$ and $y_1$, $y_2$ are positive numbers and the measure~$\mu$ on $X$ is the product of the measures $\mu_0$ on $X_0$ and
$\wrt y/y^{N+1}$ on $(0,\infty)$ for some nonnegative constant~$N$.  Li proves that either~$\cM$ is bounded on $\lp{X}$
for every $p$ in $(p_0,\infty]$, where $p_0$ depends on the parameters describing the volume of balls on $X$, or~$\cM$ is bounded only on $\ly{X}$.
The index $p_0$ is optimal in the class of the conic manifolds considered.
It is worth observing that, except for a few special cases, the metric measure spaces studied by Li have, in a way or another, unbounded geometry.
In particular, in the case where $d_{X_0}$ comes from a Riemannian metric, then, with a few trivial exceptions, the sectional curvature of $X$ is unbounded.

Related results are contained in \cite{L2}.

In these papers,  Li exhibits
examples of manifolds $M$ with cusps, where the range of $p$'s for which $\cM$ and $\cN$ are bounded on $L^p$ can be any interval of the form
$(p_0,\infty)$, or $(2p_0,\infty)$, respectively, where $p_0$ is a number in $[1,\infty)$, depending on certain parameters related to
the geometry of $M$.
Our results are reminiscent of his.

All the aforementioned results deal with specific classes of manifolds, where estimates of averages $\ds(1/\mod{B}) \, \int_B \mod{f} \wrt \mu$
can be obtained with reasonable effort as a consequence of the explicit form of the metric.

This leaves open the question of what happens on Riemannian manifolds with pinched negative curvature.

An interesting observation is that the results in \cite{LMSV} are stable under \textit{strict rough isometries}
(see Definition~\ref{def: RI} below, and \cite[Theorem~5.4]{LMSV}).   These are rough isometries in
the sense of M.~Kanai~\cite{K} that preserve the exponential rate of the volume growth of balls when the radius tends to infinity.
Rough isometries are also known as quasi-isometries (see, for instance, \cite[Section 7.2]{Gr}), and strict rough isometries are sometimes referred to
as $1$-quasi-isometries in the literature.
This suggests that it may be worth looking for sets of assumptions that are ``stable" under such maps.

{
We explore the stability of the $L^p$ boundedness for $\cM$ and $\cN$ on a class of metric measure spaces much more general than that
of Riemannian manifolds.  Indeed, we consider a pair $X$ and $X'$ of strict quasi-isometric locally doubling {geodesic} spaces enjoying the uniform ball size condition.
We show that $\cM$ and $\cN$ have the same~$L^p$ boundedness properties on $X$ and on $X'$ (see Theorem~\ref{t: quasi iso}).  

As far as stability under specific geometric constructions is concerned, we mention the case of connected sum of manifolds.  Suppose that $M$ and $N$
are Riemannian manifolds where the $L^p$ boundedness properties of either $\cM$ and $\cN$ are known.  It is natural to speculate about
the $L^p$ boundedness properties  of $\cM$ and $\cN$ on the connected sum $M \sharp N$.  We do not address the problem in full generality
and focus on the case where $M$ and $N$ are both $m$-dimensional space form with negative curvature $-a^2$ and $-b^2$.  Even in
this particular case we observe the noteworthy phenomenon that $\cM$ is bounded on $\lp{M\sharp N}$ for all $p>1$ (and it is of weak type $(1,1)$)
as it happens on $M$ and $N$, but there is an abrupt, perhaps surprising, change of $L^p$ boundedness properties as far as the uncentred operator is concerned.
Indeed, it turns out that $\cN$ is bounded on $\lp{M\sharp N}$ if and only if $p=\infty$, whereas~$\cN$ is bounded on $\lp{M}$ and on $\lp{N}$
for all $p>2$ (see Theorem~\ref{t: connected sum} below).

It may be worth to further investigate the problem for different ``summands'' $M$ and $N$.
}

Our paper is organised as follows.  Section~\ref{s: Notation} contains some basic definitions and preliminary results.  Section~\ref{s: Conformal}
is devoted to the study of the $L^p$ boundedness properties of $\cM$ and $\cN$ under a conformal change of variables.
In Section~\ref{s: Rotationally}, we study the $L^p$ bounds of $\cM$ under the quite weak assumption that the manifold has pinched negative \textit{scalar}
curvature and is rotationally symmetric.
Section~\ref{s: Strict} is devoted to the analysis of the robustness of $L^p$ bounds for $\cM$ and $\cN$ under the action of strict rough isometries.   In Section~\ref{s: Connected} we show that the operator $\cM$ on the
connected sum of two space forms with curvatures $-a^2$ and $-b^2$ is of weak type $(1,1)$, whereas the operator $\cN$ is bounded only on $L^\infty$.
Finally in Section~\ref{s: Stromberg} we analyse Str\"omberg's counterexample mentioned above.

We shall use the ``variable constant convention'', and denote by $C$ a constant that may vary from place to
place and may depend on any factor quantified (implicitly or explicitly) before its occurrence, but not on factors quantified afterwards.
The expression
$$
a(t) \asymp B(t) \qquad \forall t \in {\bD},
$$
where ${\bD} $ is some subset of the domains of $A$ and of $B,$ means that there exist (positive) constants $C$ and $C'$ such that
$$
C  A(t) \leq B(t) \leq C' A(t) \qquad \forall t \in {\bD};
$$
$C$ and $C'$ may depend on any quantifiers written {\it before} the displayed formula.

\section{Notation and preliminary results} \label{s: Notation}
Given a positive number $\al$ and a function $f$ on a Riemannian manifold $M$, we set $E_f(\al) := \bigmod{\{x\in M: \bigmod{f(x)} > \al\}}$.
Recall that the $\lorentz{p}{r}{M}$ (quasi-) norm may be defined as follows
$$
\bignorm{f}{\lorentz{p}{r}{M}}
= \Big( r \ioty \bigmod{E_f(\al)}^{r/p} \, \al^{r-1} \wrt \al\Big)^{1/r}
$$
if $1\leq r<\infty$, and $\ds\bignorm{f}{\lorentz{p}{\infty}{M}} = \sup_{\al>0} \, \al \, \bigmod{E_f(\al)}^{1/p}$.
We shall  be primarily concerned with $\lorentz{p}{1}{M}$ and $\lorentz{p}{\infty}{M}$.
For more on Lorentz spaces, see \cite[Chapter~V]{SW}.

Recall that an operator is \textit{of weak type $(p,p)$} if it is bounded from $\lp{M}$
to $\lorentz{p}{\infty}{M}$, and, in the case where $1<p<\infty$, that an operator is \textit{of restricted weak type $(p,p)$} if it is bounded from
$\lorentz{p}{1}{M}$ to $\lorentz{p}{\infty}{M}$.

In this paper, unless otherwise specified, $(M,g)$ will denote a connected noncompact complete Riemannian manifold of dimension $m$.
Suppose that $\kappa > 0$ and that $m$ is an integer $\geq 2$.  We denote by $(M_\kappa,g_\kappa)$ the $m$-dimensional simply connected Riemannian
space form with constant curvature $-\kappa^2$.

We shall need also  the following variant of $\cM_\infty$.  For any positive real number~$a$, we define
$$
\cM_\infty^\om f(x)
:= \sup_{R\geq 1} \, \frac{1}{\bigmod{B_R(x)}^{1/\om}} \, \int_{B_R(x)} \, \mod{f} \wrt V.
$$
The following result is a consequence of the results of Str\"omberg and Ionescu.  For any unexplained notation in the proof below, we refer the reader
to \cite[Section~2]{ADY}.

\begin{proposition} \label{p: endpoint}
Suppose that $\BX=G/K$ (where $G$ is a semisimple Lie group and $K$ is a maximal compact subgroup) is a symmetric space of the noncompact type and of rank one, 
and that $1<\om<2$.  Then the operator $\cM_\infty^\om$ is bounded from $\lorentz{\om}{1}{\BX}$ to $\lorentz{\om}{\infty}{\BX}$
and on $\lp{\BX}$ for all $p$ in $(\om,\om')$, where $\om'$ denotes the index conjugate to $\om$, and it is unbounded on $\lp{\BX}$ whenever $p<\om$ or $p>\om'$.
\end{proposition}

\begin{proof}
	The proof is a consequence of the sharp form of the Kunze--Stein phenomenon, proved in \cite[Theorem~1.6]{Co}; see also \cite[Theorem~1]{CMS}
	for the case of trees.  Indeed, fix $\om$ in $(1,2)$, and consider the functions $\ds\psi_R = \frac{\One_{B_R(o)}}{\bigmod{B_R(o)}^{1/\om}}$, and
$$
\Psi(x)
:= \sup_{R\geq 1} \, \psi_R(x)
	= {\min\Big[ \frac{1}{\bigmod{B_1(o)}^{1/\om}},\frac{1}{\bigmod{B_{\mod{x}}(o)}^{1/\om}}\Big]},
\quant x \in \BX.
$$
It is not hard to show that $\Psi$ belongs to $\lorentz{\om}{\infty}{{\BX}}$
and that  $\bignorm{\psi_R}{\lorentz{\om}{1}{\BX}} = 1$ for every $R>0$ .  Now,
$$
f*\psi_R (x)
= \int_G \, f(y) \, \psi_R(y^{-1}x) \wrt y
= \frac{1}{\bigmod{B_R(o)}^{1/\om}} \, \int_{B_R(x)} \, f(y) \wrt y.
$$
It follows that
$$
\cM_\infty^\om f(x) \leq \mod{f}*\Psi(x).
$$
By \cite[Theorem~1.6]{Co}, $\lorentz{\om}{1}{G}*\lorentz{\om}{\infty}{G} \subseteq \lorentz{\om}{\infty}{G}$ 
In particular, if
$f$ is $K$--right-invariant, then so is $f*\Psi$, whence there exists
a constant $C$ such that
$$
\bignorm{\cM_\infty^\om f}{\lorentz{\om}{\infty}{\BX}}
\leq C \, \bignorm{f}{\lorentz{\om}{1}{\BX}}.
$$

Since $\Psi$ is $K$-invariant on $\BX$, and convolution with $\Psi$ is bounded from
$\lorentz{\om}{1}{\BX}$ to $\lorentz{\om}{\infty}{\BX}$, and on $\lp{\BX}$ for all $p$ in $(\om,2]$,
convolution with $\Psi$ is bounded on $\laq{\BX}$ for all $q$ in $[2, \om')$,
as required to conclude the proof of the positive result.

It remains to prove that $\cM_\infty^\om$ is unbounded when $p<\om$ or $p>\om'$.  Fix $p>\om'$ and choose $q$ in the interval $(\om',p)$.
Write $\de(q)$ in place of $\mod{1-2/q}$ for short.  The spherical function $\vp_{i\de(q)\rho}$ is in $\lp{\BX}$ 
	We estimate $\cM_\infty^\om\vp_{i\de(q)\rho}$.
	Preliminarily, we estimate the spherical Fourier transform of $\One_{B_R(o)}$ at the point $i\de(q)\rho$.  Observe that 
$$
\wt\One_{B_R(o)} (i\de(q)\rho)
= \int_{B_R(o)} \vp_{i\de(q)\rho} \wrt \mu.
$$
Recall the estimate $\vp_{i\de(q)\rho} (x) \asymp \e^{-(2/q)\rho \mod{x}}$ when $\mod{x}$
is large.  Then
$$
\int_{B_R(o)} \vp_{i\de(q)\rho} \wrt \mu
\asymp \int_0^R \e^{2 (1-1/q)\rho s} \wrt s
\asymp \e^{2 (1-1/q)\rho R}
$$
as $R$ tends to $\infty$.  Now, observe that
$$
\vp_{i\de(q)\rho}*\psi_R
=\frac{\wt\One_{B_R(o)} (i\de(q)\rho)}{\mu(B_R(o))^{1/\om}} \, \vp_{i\de(q) \rho}
\asymp \e^{2 (1-(1/q)-(1/\om))\rho R} \, \vp_{i\de(q) \rho}.
$$
Notice that $\ds \frac{1}{q}+\frac{1}{\om} < 1$, whence
$$
\cM_\infty^\om \vp_{i\de(q)\rho}
\geq C \,  \sup_{R\geq 1} \, \e^{2 (1-(1/q)-(1/\om))\rho R} \, \vp_{i\de(q) \rho}
= \infty
$$
for every $x$ large enough.  Thus, $\cM_\infty^\om$ is unbounded on $\lp{\BX}$ as required.

Now, suppose that $p<\om$.  We shall show that $\cM_\infty^\om \One_{B_1(o)}$
is not in $\lp{\BX}$.  For~$\mod{x}$ large,
$$
\cM_\infty^\om \One_{B_1(o)} (x)
=     \mu\big(B_1(o)\big)\, \mu\big(B_{\mod{x}+1}(x)\big)^{-1/\om}
\asymp \e^{-(2\rho/\om) \mod{x}},
$$
whence
$$
\bignormto{\cM_\infty^\om \One_{B_1(o)}}{\lp{\BX}}{p}
\leq C \, \int_0^\infty \e^{-2\rho(1-p/\om) s} \wrt s,
$$
which is equal to $\infty$, because $p<\om$, as required.
\end{proof}

\section{Conformal case} \label{s: Conformal}

In this section we consider the problem of the ``stability" of the $L^p$ boundedness properties of the maximal operators $\cM$ and $\cN$ on
Riemannian manifolds under perturbations of the metric.  An interesting case arise in connection to conformal changes of the metric.
It may be worth observing that multiplying a metric by a conformal factor may affect the bounds of the sectional curvatures significantly.  In particular,
the perturbed metric may very well have unbounded sectional curvature even though the original metric is assumed to have pinched negative curvature.

Suppose that $(M,g_1)$ is an $m$-dimensional complete Riemannian manifold.
Suppose that $\la$ is a smooth positive function on $M$, and consider the metric $g := \la^2 \, g_1$ on $M$.  Denote by $\mu_1$ and $\mu$
the Riemannian measures associated to $g_1$ and $g$, respectively.  For $x$ in~$M$ and $R>0$ denote by $B_R(x)$ and $B_R^1(x)$ the geodesic balls
with centre $x$ and radius $R$ with respect to $g$ and~$g_1$, respectively.

The next result extends and generalises Str\"omberg's result \cite[p.~126]{Str} mentioned in the introduction.

\begin{theorem} \label{t: conformal}
	Consider the Riemannian manifold $(M,g_1)$ as above.  Assume that the conformal factor $\la$ satisfies the
	estimate $1\leq \la\leq \om$ for some real number $\om$, and that there exists a positive constant $c$ such that
	\begin{equation} \label{f: lower mult}
	\mu_1 \big(B^1_{R/\om}(x)\big)
	\geq c\, \mu_1\big(B^1_{R}(x)\big)^{1/\om}
	\quant x \in M \quant R\geq 1.
\end{equation}
The following hold:
	\begin{enumerate}
		\item[\itemno1]
			if $\cM_\infty^\om$ is bounded on $\lp{\mu_1}$ [resp. it restricted weak type $(p,p)$ for some $p>1$, resp. it is of weak type $(1,1)$],
			then~$\cM$ is bounded on $L^p(\mu)$ [resp. it is of restricted weak type $p$ for some $p>1$, resp. of weak type $(1,1)$];
		\item[\itemno2]
			if $a\leq b <2a$, $\om = b/a$, $M$ {is a symmetric space of the noncompact type and of rank one,} 
			and $g_1$ {denotes the standard $G$ invariant metric}, then $\cM$ is bounded on $\lp{\mu}$ for all $p>\om$ and
			it is of restricted weak type $(\om,\om)$.
	\end{enumerate}
\end{theorem}

\begin{remark} \label{t: exponential}
We note that condition  \eqref{f: lower mult} is guaranteed assuming a purely exponential volume growth, namely,
\[
	C^{-1} \exp(c R) \leq |B_R|\leq C \exp(c R),\ C>1,\ c>0.
 \]	
\end{remark} 

\begin{proof}
{ Clearly, only the case where $\omega>1$ requires proof.}	Notice  also that \rmii\ follows directly from \rmi\ and Proposition~\ref{p: endpoint}.

	Thus, it remains to prove \rmi.  First we consider the strong type estimate.  Observe that $(M,g)$ is a complete Riemannian manifold.
	Define $\cI_{p}: \lp{\mu} \to \lp{\mu_1}$ by
	$$
	\cI_{p}f
	= \la^{m/p}\, f
	\quant f \in \lp{\mu}.
	$$
	Since $\wrt \mu = \la^m \wrt \mu_1$, the operator $\cI_{p}$ is an isometric isomorphism between $\lp{\mu}$ and $\lp{\mu_1}$.
	Observe that
	\begin{align*}
	\frac{\la(x)^{m/p}}{\mu\big(B_R(x)\big)}\, \int_{B_R(x)} \mod{f} \wrt \mu
	& =    \frac{\lambda(x)^{m/p}}{\mu\big(B_R(x)\big)}\, \int_{B_R(x)} \la^{m/p'} \, \mod{\cI_{p}f} \wrt \mu_1 \\
	& \leq  \frac{\om^m}{\mu\big(B_R(x)\big)}\, \int_{B_R(x)} \mod{\cI_{p}f} \wrt \mu_1.
	\end{align*}
	Define
	$$
	g_\om := \om^2\, g_1,
	$$
	and denote by $B_R^\om(x)$ the geodesic ball with centre $x$ and radius $R$ with respect to the metric $g_\om$.
	Notice that $g=\la^2\, g_1\leq g_\om$ in the sense of quadratic forms, whence $B^\om_R(x)\subset B_R(x)$.  Furthermore the Riemannian measure $\mu_\om$
	associated to $g_\om$ is just $\om^m\, \mu_1$.  Now,
	\[
	\mu\big(B_R(x)\big)
	\geq \mu\big(B^\om_R(x)\big)
	\geq \mu_1\big(B^\om_R(x)\big),
	\]
	where the last inequality follows from the assumption $\la \geq 1$, which implies that $\mu\geq \mu_1$.
	Denote by $d_\om$ and $d_1$ the distances associated to $g_\om$ and $g_1$, respectively.  Notice that $B^\om_R(x)=B^1_{R/\om}(x)$, because
	\[
	d_\om(x,y)=\om\, d_1(x,y)\quant (x,y)\in M\times M.
	\]
	Then
	\[
	\mu_1\big(B^\om_R(x)\big)
	= \mu_1 \big(B^1_{R/\om}(x)\big)
	\geq c\, \mu_1\big(B^1_{R}(x)\big)^{1/\om};
	\]
	the inequality above follows from \eqref{f: lower mult}.  Altogether, we find that
	$$
	\frac{\la(x)^{m/p}}{\mu\big(B_R(x)\big)}\, \int_{B_R(x)} \mod{f} \wrt \mu
	\leq \frac{\om^m}{c\,\mu_1\big(B_{R}^1(x)\big)^{1/\om}}\, \int_{B_R^1(x)} \mod{\cI_{p}f} \wrt \mu_1;
	$$
	the trivial containment $B_R(x)\subseteq B_R^1(x)$ has been used in the last inequality.
	By taking the supremum over $R\geq 1$ of both sides, we obtain that
	\begin{equation} \label{f: comparison maximal}
	\big(\cI_{p} \cM_\infty f\big) (x)
	\leq C \, \cM_\infty^{\om} \big(\cI_{p}f\big)(x),
	\end{equation}
	where the maximal function on the left hand side is with respect to the metric $g$, and that on the right hand side is with respect to the metric $g_1$.  Hence
	$$
	\bignorm{\cM_\infty f}{\lp{\mu}}
	=   \bignorm{\cI_{p} \cM_\infty f}{\lp{\mu_1}}
	\leq C \,\bignorm{\cM_\infty^{\om}\big(\cI_{p}f\big)}{\lp{\mu_1}}.
	$$
	The assumption that $\cM_\infty^{\om}$ is bounded on $\lp{\mu_1}$ and the fact that $\cI_p$ is an isometry between $\lp{\mu}$ and $\lp{\mu_1}$ imply that
	$$
	\bignorm{\cM_\infty^{\om}\big(\cI_{p}f\big)}{\lp{\mu_1}}
	\leq C \,\bignorm{\cI_{p}f}{\lp{\mu_1}}
	=    C \, \bignorm{f}{\lp{\mu}}.
	$$
	Combining the estimates above yields
	$\bignorm{\cM_\infty f}{\lp{\mu}}\leq C \,\bignorm{f}{\lp{\mu}}$, as required.

	The endpoint results follows by a similar argument.  Here are the details in the case where we assume a restricted weak type estimate on $(M,g_1)$.
	It is straightforward to check that
	$$
	\bignorm{g}{\lorentz{p}{1}{\mu_1}} 	
	\leq \bignorm{g}{\lorentz{p}{1}{\mu}} 	
	\leq \om^{m/p}\, \bignorm{g}{\lorentz{p}{1}{\mu_1}},
	$$
	and that a similar estimate holds for weak-$L^p$ (quasi-) norms.
	Now, from this consideration, the bounds of the conformal factor $\la$
	and \eqref{f: comparison maximal} we get that
	$$
	\begin{aligned}
	\bignorm{\cM_\infty f}{\lorentz{p}{\infty}{\mu}}
		& \leq \bignorm{\cI_p\cM_\infty f}{\lorentz{p}{\infty}{\mu}} \\
		& \leq \om^{m/p} \, \bignorm{\cI_p\cM_\infty f}{\lorentz{p}{\infty}{\mu_1}} \\
		& \leq C\, \om^{m/p} \, \bignorm{\cM_\infty^{{\om}} \cI_pf}{\lorentz{p}{\infty}{\mu_1}} \\
		& \leq C\, \om^{m/p} \, \bigopnorm{\cM_\infty^{{\om}} }{\lorentz{p}{1}{\mu_1};\lorentz{p}{\infty}{\mu_1}} \bignorm{\cI_pf}{\lorentz{p}{1}{\mu_1}} \\
		& \leq C\, \om^{2m/p} \, \bigopnorm{\cM_\infty^{{\om}} }{\lorentz{p}{1}{\mu_1};\lorentz{p}{\infty}{\mu_1}} \bignorm{f}{\lorentz{p}{1}{\mu_1}},
	\end{aligned}
	$$
	as required.
\end{proof}

{In the next proposition we show that the pointwise bound $1\leq \la \leq \om$ appearing in the statement of Theorem~\ref{t: conformal} 
can be derived from appropriate geometric assumptions on the original manifold $(M,g_1)$ and the conformally equivalent metric $\la^2\, g_1$.

\begin{proposition} \label{p: conformal}
Let $a,b,\omega$ be positive constant such that $0<a\le b \le \omega a$. Suppose that $(M,g_1)$ is an $m$-dimensional complete Riemannian manifold with $m\geq 3$, that $\la$ is a smooth positive function on $M$, and consider the metric 
$g := \la^2 \, g_1$ on $M$.  If
\begin{equation}\label{eq:lower Ricci}
	-(m-1) \omega^2a^2 g_1\leq \operatorname{Ric_{g_1}} \leq -(m-1) b^2 g_1,
\end{equation}
the scalar curvature satisfies the bounds 
\begin{equation}\label{eq:scalar}
-m(m-1)b^{2} \leq 
	\operatorname{Scal}_{ g} \leq -m(m-1) a^{2}
\end{equation}
and either
\begin{equation}\label{eq:condition1}
 \inf_{M} \la >0
\end{equation}
or
\begin{equation}\label{eq:condition2}
	(M,g)  \text{ is complete with } \operatorname{Ric}_{g} \geq -(m-1)b^{2},
\end{equation}
then the conformal factor $\la$ satisfies the pointwise bound
\begin{equation}\label{eq:estimate-conf}
 1 \leq \lambda\leq \omega.
\end{equation}
\end{proposition}

\begin{proof}
	Set $\la = u^{2/(m-2)}$.  Then $u$ satisfies the Yamabe equation
	\begin{equation}\label{eq:yamabe}
	c_{m}  \Delta u  = \operatorname{Scal}_{g_1} u - \operatorname{Scal}_{g} u^{(m+2)/(m-2)} 
	\end{equation}
	The upper scalar curvature estimate
	\[
	\operatorname{Scal}_{g} \leq -m(m-1)a^{2}
	\]
	and the lower Ricci curvature estimate \eqref{eq:lower Ricci} imply that
	\begin{equation}
 	\Delta u \geq f(u) := c'_{m} \, \big(-\omega^2a^2 u +a^2 u^{(m+2)/(m-2)}\big).
	\end{equation}
	Since, by the Ricci curvature condition, $M$ satisfies the full Omori-Yau maximum principle and  $f(t) \approx t^{\si}$ as $t \to + \infty$ for
	some $\si >1$, the Cheng-Yau and Motomiya a priori estimates yield $\sup_{M}u = u^{\ast} < +\infty$ and $f(u^{\ast}) \leq 0$; see e.g. 
	\cite[Theorem 1.31 and Example 1.13]{PRS}.  This gives the upper estimate of the conformal factor:
	\[
	\la^{2} = u^{4/(m-2)}  \leq \frac{\omega^2a^2}{a^{2}}= \omega^2.
	\]
	In order to get the lower estimate of the conformal factor, we start by noting that, thanks to the assumption
	\[
	\operatorname{Scal}_{g}  \geq -m(m-1)b^{2}
	\]
	and the upper Ricci curvature estimate \eqref{eq:lower Ricci} 
	it holds
	\begin{equation}
	\Delta u \leq f(u) := c'_{m} \, \big(-b^2 u +b^2 u^{(m+2)/(m-2)}\big).
	\end{equation}
	Thus, if we  assume  that \eqref{eq:condition1} is met, using the weak Omori-Yau in the form of a minimum principle at infinity we obtain that 
	$\inf_{M} u = u_{\ast} >0$ satisfies $f(u_{\ast}) \geq 0$. It then follows that
	\[
	\lambda^{2} = u^{4/(m-2)} \geq \frac{b^2}{b^{2}}= 1.
	\]
	If instead we assume the validity of \eqref{eq:condition2} then the complete manifold $(M,g)$ enjoys the full Omori-Yau maximum principle. 
	Therefore we can apply the first part of the proof to the conformal factor $\la^{-1}= u_{1}^{2/(m-2)} $ and get an upper estimate for 
	$u_{1}$ that implies the desired conclusion.
\end{proof}

\begin{remark}
	The assumptions \eqref{eq:lower Ricci} and \eqref{eq:scalar} are not incompatible. Namely, given a complete manifold $(M,g_1)$ satisfying \eqref{eq:lower Ricci}, there exists a positive solution of \eqref{eq:yamabe}, and thus a conformal deformation $g=\lambda^2 g_1$ of $g_1$ which verifies \eqref{eq:scalar}; see \cite[Theorem 6.15]{MRS}.
	
	An example of application of Proposition \ref{p: conformal} is the following. Suppose that $a\leq b <2a$, $\om = b/a$, and let $M$ be a symmetric space of the noncompact type and of rank one, 
	with $g_1$ the standard $G$ invariant metric. $(M,g_1)$ is Einstein, and up to rescaling we can suppose that $\operatorname{Ric}_{g_1} =-(m-1)b^2 g_1$. Then, given any conformal deformation $(M,g)$ of $(M,g_1)$, $g_1=\lambda g^2$ such that
	\begin{equation*}
		-m(m-1)b^{2} \leq 
		\operatorname{Scal}_{ g} \leq -m(m-1) \omega^{-2}b^{2},
	\end{equation*}
	if either \eqref{eq:condition1} or \eqref{eq:condition2} are satisfied then $1\le\lambda\le\omega$ and by Theorem \eqref{t: conformal} 
$\cM$ is bounded on $\lp{\mu}$ for all $p>\om$ and
	it is of restricted weak type $(\om,\om)$.
\end{remark}
}

\section{Rotationally symmetric manifolds with pinched negative scalar curvature} \label{s: Rotationally}

We say that a complete noncompact $m$-dimensional Riemannian manifold $(M,g)$ is a \textit{model manifold} if $M$ is diffeomorphic to $\BR^n$
and its metric $g$ written in polar cordinates $(t,u)$ (here $t$ positive and $u$ belongs to $\BS^{m-1}$) is of the form
\begin{equation} \label{f: metric rotationally}
g = \wrt t^2+ j(t)^2\,g_{\BS^{m-1}}(u).
\end{equation}
The warping function $j$ is assumed to be smooth on $[0,\infty)$ and strictly positive on $(0,\infty)$, and to satisfy the conditions
\begin{equation} \label{f: metric rotationally II}
j'(0)=1,\qquad\mathrm{and}\qquad j^{(2k)}(0)=0 \quant k\in\BN.
\end{equation}
In particular, we see that $j(t) \sim t$ as $t$ tends to $0$.

\begin{theorem} \label{t: scalar}
Suppose that $m\geq 3$ and that $M$ is a $m$-dimensional complete noncompact model manifold with pinched negative scalar curvature, i.e.
\[
-m(m-1)\, b^2\leq \mathrm{Scal}_g \leq - m(m-1)\, a^2,
\]
for some positive numbers $a$ and $b$ such that $a<b<2a$.
Then $\cM$ is bounded on $\lp{M}$ for all $p>b/a$, and it is bounded from $\lorentz{b/a}{1}{M}$ to $\lorentz{b/a}{\infty}{M}$.
\end{theorem}

\begin{proof}
	Consider the metric $g$ of $M$ in polar form \eqref{f: metric rotationally}, with $j$ satisfying \eqref{f: metric rotationally II}.

	We \textit{claim} that $1/j$ is integrable in a neighbourood of $+\infty$.
	We argue by contradiction.  Suppose that $1/j$ is nonintegrable in a neighbourhood of~$+\infty$.  Define
	\[
	\phi(t):= \exp\Big(\int_{1}^t \frac{dq}{j(q)} \Big)
	\quant t \in (0,\infty).
	\]
	Clearly $\phi$ is a smooth, strictly increasing function on $(0,+\infty)$.  Furthermore~$\phi$ diverges to infinity as $t$ tends to $+\infty$
	and tends to $0$ as $t$ tends to $0$ (because $1/j$ is nonintegrable in a right neighbourhood of $0$).  Thus, $\phi$ is a diffeomorphism of $(0,\infty)$.
	Note that
	\[
	\phi(t)^2
	= \phi'(t)^2\, j(t)^2,
	\]
	so that
	\begin{align*}
	\wrt t^2+ j(t)^2\, g_{\BS^{m-1}}(u)
	& = \frac{1}{\phi'(t)^2} \left[\phi'(t)^2 \wrt t^2 + \phi(t)^2 \, g_{\BS^{m-1}}(u) \right] \\
	& = \frac{1}{\phi'\big(\phi^{-1}(\rho)\big)^2}\left[\wrt \rho^2 + \rho^2 \, g_{\BS^{m-1}}(u) \right];
	\end{align*}
	the last equality follows from the change of variable $\rho:=\phi(t)$.  Thus, $(M,g)$ is isometric to $(\BR^m, {\la_0}^2\, g_{Eucl})$ with conformal factor
	given by
	\[
	{\la_0}(\rho,u)
	:= \frac{1}{\phi'\big(\phi^{-1}(\rho)\big)}.
	\]
	Now, a nonexistence result for conformal metrics due to Ni \cite[Theorem 6.4]{Ni} contradicts the assumption that $\mathrm{Scal}_g\leq -m(m-1)\, a^2<0$,
	and the claim is proved.

	Next, set $A:= \exp\big[-\int_1^\infty (1/j) \wrt m\big]$, where $m$ denotes the Lebesgue measure, and observe that
	$A\, \phi(t) = \exp\big[-\int_{t}^\infty (1/j) \wrt m\big]$.
	Clearly $A\, \phi$ is a diffeomorphism between $(0,+\infty)$ onto $(0,1)$.  Set $\psi := A \, \phi$.  By arguing much as above, the change of variables
	$\rho:=\psi(t)$ yields
	\begin{align*}
	\wrt t^2+ j(t)^2\, g_{\BS^{m-1}}(u)
	& = \frac{1}{\psi'\big(\psi^{-1}(\rho)\big)^2}\left[\wrt \rho^2 + \rho^2 \, g_{\BS^{m-1}}(u) \right].
	\end{align*}
	Thus, $(M,g)$ is isometric to $\big(\BB_1, { \lambda}^2 g_{Eucl}\big)$ with conformal factor
	\[
{ \lambda}(\rho,u)=\frac{1}{\psi'\big(\psi^{-1}(\rho)\big)}.
	\]
	Here and on, $\BB_R$ denotes the $m$-dimensional Euclidean ball with center the origin and of radius $R$.
	We adapt to our assumptions the idea of the proof of \cite[Theorem A]{Ahl}. Define the function
	$\La:\BB_1 \to \BR$ by $\La(\rho,u):= {\lambda(\rho,u)^{(m-2)/2}}$. By the conformal change rule for the scalar curvature,
	we have that $\La$ satisfies the Yamabe equation
	\[
		\Delta \La = -\frac{m-2}{4(m-1)}\, \mathrm{Scal}_g \, \La^{\nu(m)}.
	\]
	where $\ds \nu(m) := \frac{m+2}{m-2}$.
	For any $R>0$ and $c>0$, denote by $\La_{c,R}$ the rotationally symmetric function on $\BB_R$, defined by
	\[
	\La_{c,R}(\rho, u)
	= \Big(\frac{2}{c}\, \frac{R}{R^2-\rho^2}\Big)^{(m-2)/2}.
	\]
	A standard computation  in polar co-ordinates shows that
	\[
	\Delta \La_{c,R}
	= \frac{m(m-2)}{4}\, c^2 \La_{c,R}^{\nu(m)}
	\]
	on $\BB_R$.  Now, if $R<1$, then
	\begin{equation} \label{f: La minus LaaR}
	\Delta (\La - \La_{a,R})
		= -\frac{m-2}{4(m-1)}\, \big(\mathrm{Scal}_g\, \La^{\nu(m)}+m(m-1)\, a^2\La_{a,R}^{\nu(m)} \big).
	\end{equation}
	Set $E:=\{(\rho,u)\in \BB_R\,:\,\La(\rho, u) > \La_{a,R}(\rho, u)\}$.

	On the one hand \eqref{f: La minus LaaR} and the assumption $\mathrm{Scal}_g\leq -m(m-1)\, a^2$ imply that $\Delta(\La - \La_{a,R})>0$ on $E$.

	On the other hand $\La - \La_{a,R}$ tends to $-\infty$ as $\rho$ tends to $R$ from the left, so that $E\subset  \subset \BB_R$.  By the strong maximum principle,
	we have thus that $E$ is empty and $\La \leq \La_{a,R}$ in $\BB_R$.  Taking the limit as $R$ tends to $1$, we obtain that
	\begin{equation}\label{f: upp est}
	\la \leq \frac{2}{a}\frac{1}{(1-\rho^2)}
	\end{equation}
	on $\BB_1$.

	Similarly, suppose that $R>1$.  Then
	\begin{equation} \label{f: eq I}
		\Delta \big(\La - \La_{b,R})
		= -\frac{(m-2)}{4(m-1)}\, \big(\mathrm{Scal}_g\La^{\nu(m)}+m(m-1)\,b^2\, \La_{b,R}^{\nu(m)} \big)
	\end{equation}
	on $\BB_{{1}}$.  Since $R>1$, the function $\La_{b,R}$ is bounded on $\overline {\BB}_1$.  Let $\{\rho_k\}$ be a sequence of numbers in $(0,1)$ such that
	$$
	\lim_{k\to\infty}\rho_k=1
	\quad\hbox{and}\quad
	\lim_{k\to\infty}\,  {j\big(\psi^{-1}(\rho_k)\big)}=+\infty.
	$$
Such a sequence clearly exists because of the integrability assumption of~$1/j$.   It is straightforward to check that
	$$
	\lim_{k\to \infty} \, \psi'(\psi^{-1}(\rho_k)) = 0
	\quad\hbox{and}\quad
	\lim_{k\to \infty} \,\La(\rho_k,u) = \infty,
	$$
	so that $\La(\rho_k,u) - \La_{b,R}(\rho_k,u) >0$ for $k$ large enough.
	Set
	\[
		F_k
		:= \{(\rho,u)\in (0,\rho_k) \times \BS^{m-1}:\,\La(\rho, u)<\La_{b,R}(\rho, u)\}\subset\subset \BB_{\rho_k}.
	\]
	From \eqref{f: eq I} and the assumption $\mathrm{Scal}_g\geq -m(m-1)\, b^2$, we deduce that $\Delta(\La - \La_{\cg{b},R})<0$ on $F_k$.
	By the strong maximum principle, $F_k$ is empty and $\La \geq \La_{b,R}$ in $\BB_{\rho_k}$.  Taking the limit as $k$ tends to $\infty$,  we have that
	$\La \geq \La_{b,R}$ in $\BB_{1}$.  Then taking the limit as $R$ tends to $1$, we see that
	\begin{equation}\label{f: low est}
	\la
		\geq
		\la_b
		:=\frac{2}{b}\frac{1}{(1-\rho^2)}
	\end{equation}
	on $\BB_1$.
	Note that $(\BB_1, \la_b^2 \,g_{Eucl})$ is isometric to $(M_b,g_b)$.  Hence $(M,g)$ is isometric to $(M_b,{(\la/\la_b)} ^2\, g_b)$ and, because of
	\eqref{f: upp est} and \eqref{f: low est}, the following estimate holds on $M_b$:
	\[
	1 \leq {\frac{ \la}{\la_b}} \leq \frac{b}{a}.
	\]
	The required conclusion now follows from Theorem \ref{t: conformal}.
\end{proof}

The next example shows that the assumptions of Theorem \ref{t: scalar} do not imply that the sectional curvature of the manifold at hand is negatively pinched.  

Assume that $m$ is at least $3$.  Recall that $\wrt t^2+ \sinh^2(t)\, g_{\BS^{m-1}}$ is the metric of the hyperbolic space of constant curvature $-1$ 
written in polar coordinates.  The example we want to exhibit is a model manifold with a perturbed hyperbolic metric of the form 
\begin{equation} \label{f: model mod hyp}
\wrt t^2+j_\tau(t)^2g_{\BS^{m-1}}, 
\quad\hbox{where}\quad
j_\tau(t) = \sinh(t) + \psi_\tau(t).
\end{equation}
Here $\tau$ is a large positive number and $\psi_\tau$ is a smooth real function on $[0, \infty)$ that meets the following requirements:
\begin{enumerate}
	\item[\itemno1]
		$\psi_\tau$ vanishes off the interval $[\tau-\nu,\tau+\nu]$;
	\item[\itemno2]
		$\psi_\tau$ satisfies the following estimates 
			\[
			\max \big\{\bignorm{\psi_\tau}{\ly{[0,\infty)}}, \bignorm{\psi_\tau'}{\ly{[0,\infty)}} \big\} 
			\le \de,\quad \psi_\tau''(t)
			\le \de,
			\quad \text{and}\quad \min \psi_\tau''=-T,
			\]
		where $\de$ and $\nu$ are constants in $(0,1)$ and in $\big(0,1/2\big)$, respectively, that will be specified later, 
		$\ds T:=\frac{c_m}{\e^{\nu}} \, \sinh(\tau+\nu)$ with $\ds c_m\in \Big(1,\frac{3m}{8} \Big)$.
\end{enumerate}
A convenient way to construct $\psi_\tau$ is as follows.  Start with a function $\phi_\tau: \BR\to \BR$ enjoying the following properties:	
	\begin{enumerate}
		\item[(a)]  $\phi_\tau(\tau+\cdot )$ is an even function in $C^{\infty}_c([-\nu,+\nu])$ and
		$$
			\int_{\tau-\nu}^{\tau+\nu} \phi_\tau(t)\wrt t
			= 0
			= \int_{\tau}^{\tau+\nu} \phi_\tau(t)\wrt t;
		$$
		\item[(b)] $\ds \max_{[\tau-\nu,\tau+\nu]}\,\phi_\tau(t) \le \de$ and $\ds \min_{[\tau-\nu,\tau+\nu]}\, \phi_\tau(t) = -T$;
		\item[(c)] $\ds\int_{\tau-\nu}^{\tau+\nu} |\phi_\tau(t) |\wrt t \le \de$.
	\end{enumerate}

\begin{figure}[h]
{	\centering
	\includegraphics[width=1\textwidth]{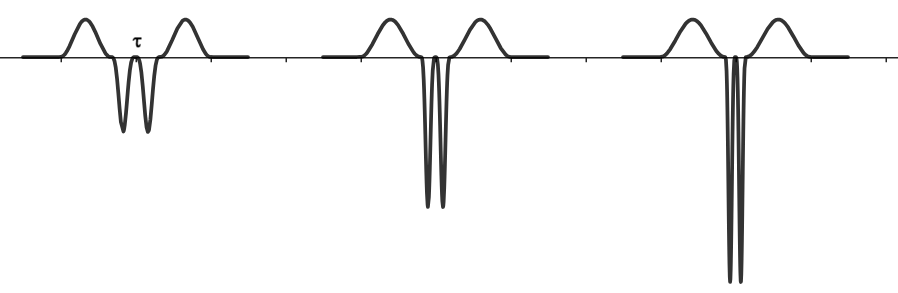}}
	\caption{Examples of functions $\phi_\tau$ for increasing values of $\tau$ and $T$ and constant value of $\delta$.}

\end{figure}

Then define
\begin{align*}
	\psi_\tau(t)
	= \int_0^t \wrt s \int_0^s \phi_\tau(q)\wrt q.
\end{align*}
We show that $\psi_\tau$ has the required properties.
Observe that	
	\[
		\int_0^\tau \phi_\tau(q) \wrt q
		= \int_{\tau-\nu}^\tau \phi_\tau(q) \wrt q=0,
	\]
	and that $\ds \int_0^s \phi_\tau(q)\wrt q=0$ if either $s>\tau+\nu$ or $s<\tau-\nu$.
	Moreover, if $\eta\in[0,\nu]$, then  
	\begin{align*}
	\int_0^{\tau-\eta}\!\! \phi_\tau(q)\wrt q
		= \int_{\tau-\nu}^{\tau-\eta}\!\! \phi_\tau(q)\wrt q= - \int_{\tau-\eta}^\tau \!\!\phi_\tau(q)\wrt q 
		= -\int_\tau^{\tau+\eta} \!\!\phi_\tau(q)\wrt q= -\int_0^{\tau+\eta} \!\!\phi_\tau(q)\wrt q
	\end{align*}
	i.e., the function $\ds \eta\to \int_0^{\tau+\eta} \phi_\tau(q)\wrt q$ is an odd function.  Furthermore
	\[ 
	\int_0^t\wrt s \int_0^s\phi_\tau(q)\wrt q
	= 0,
	\quant t\in [\tau-\nu,\tau+\nu]^c,
	\] 
	and 
	\[
		\Big|\int_0^t\wrt s \int_0^s\phi_\tau(q)\wrt q\Big|
		\le 2\, \de \nu
		< \de,
		\quant t\in [\tau-\nu,\tau+\nu].
		\]
	
	Now we come back to the model manifold with metric defined in \eqref{f: model mod hyp}	
	Note that this metric is hyperbolic when $t\not\in [\tau-\nu,\tau+\nu]$. Hence, in the following we will only consider $t\in [\tau-\nu,\tau+\nu]$.
	Recall that the sectional curvatures are $-j_\tau''(t)/j_\tau(t)$ for tangent planes containing the radial direction and 
	$(1-(j_\tau'(t))^2)/j_\tau(t)^2$ for planes orthogonal to the radial direction. Accordingly, the scalar curvature is
	\[
	\mathrm{Scal}(t)=(m-1)(m-2)\frac{1-(j_\tau'(t))^2}{j_\tau(t)^2}-2(m-1)\frac{j_\tau''(t)}{j_\tau(t)}.
		\]
We estimate
\begin{align*}
\frac{j_\tau''(t)}{j_\tau(t)}\ge \frac{\sinh''(t)-T}{\sinh(t)+\psi_\tau(t)} 
	=\frac{\sinh''(t)-c_m\,\e^{-\nu}\sinh(\tau+\nu)}{\sinh(t)+\psi_\tau(t)},
\end{align*}
where the equality is attained at some point $t_0$ in $(\tau,\tau+\nu)$.

Observe that
\[	
1 \le \frac{\sinh( \tau +\nu)}{e^\nu\sinh( \tau)} \le 1+\frac{1}{2e^{2\tau}}
\quant\,\nu>0,\,\tau>0.
\]
Next, fix $\vep>0$; by choosing $\nu$ and $\delta$ small enough and $\tau$ large enough (depending on~$\vep$), we see that 
	\begin{align*}
	\frac{j_\tau''(t)}{j_\tau(t)}
		 \ge \frac{\sinh''(t)- c_m\, \e^{-\nu} \, \sinh(\tau+\nu)}{\sinh(t)+\psi_\tau(t)}
		 =   \frac{1- c_m\,  \ds\frac{\sinh(\tau+\nu)}{\e^{\nu} \,\sinh(t)}}{1+\ds\frac{\psi_\tau(t)}{\sinh(t)}} 
		 \ge \frac{1- c_m\,  \ds\frac{\sinh(\tau+\nu)}{\e^{\nu} \,\sinh(\tau)}}{1+\ds\frac{\delta}{\sinh(\tau-\nu)}},
	\end{align*}
	which is clearly $\ge 1-c_m-\vep$.
In the same vein
\begin{align*}
	\frac{j_\tau''(t)}{j_\tau(t)}\le \frac{\sinh''(t)+\delta}{\sinh(t)-\delta}\le 1+\vep.\end{align*}
Similarly,
\begin{align*}
	\frac{1-(j_\tau'(t))^2}{j_\tau(t)^2}&\le \frac{1-(\sinh'(t))^2-2\sinh'(t)\psi'_\tau(t)}{\sinh^2(t)+2\sinh(t)\psi_\tau(t)}\\
		&\le  \frac{1-(\sinh'(t))^2}{\sinh^2(t)+2\delta \sinh(t)}+\frac{2\delta\sinh'(t)}{\sinh^2(t)+2\delta \sinh(t)} \\
		& \le -1 + \vep,
\end{align*}
and 
\begin{align*}
	\frac{1-(j_\tau'(t))^2}{j_\tau(t)^2}&\ge \frac{1-(\sinh'(t))^2-2\sinh'(t)\delta - \delta^2}{\sinh^2(t)+2\delta\sinh(t) + \delta^2}\ge -1 - \vep ,
\end{align*}
so that
\[
-m(1+\vep)\le \frac{\mathrm{Scal}(t)}{m-1}\le -\left[(m-2)(1-\vep) - 2 (c_m-1+\vep)\right]
,
\]
i.e., that the scalar curvature is pinched between two negative constants whose ratio is smaller than $4$. However,
\[
	\frac{j_\tau''(t_0)}{j_\tau(t_0)}
	=    \frac{1-c_m\, \e^{-\nu}\,s(\tau+\nu)/s(t)}{1+\psi_\tau(t_0)/s(t_0)}
	\le  \frac{1-c_m\, \e^{-\nu}}{1-\delta/s(\tau)}
	< 0
	\]
for $\nu,\delta$ small enough, so that the sectional curvatures of the manifold attain positive values.
pinched.

\section{Strict rough isometries} \label{s: Strict}

In this section $(X,d,\mu)$ will denote a measured metric space, and $\cB$ the family of all balls on $X$.  We assume that $0< \mu(B) <\infty$
for every $B$ in~$\cB$.  For each $B$ in $\cB$ we denote by $c_B$ and $r_B$ the centre and the radius of $B$, respectively.
Furthermore, we denote by $k B$ the ball with centre $c_B$ and radius $k r_B$.  For each $s$ in $\BR^+$,
we denote by $\cB_s$ the family of all balls~$B$ in~$\cB$ such that $r_B \leq s$, and by $\cB^s$ the collection of all balls in $X$ with radius $s$. In the main results of this section, we will require the space $X$ to be geodesic, i.e., for every $x,y\in X$ there exists (at least) one geodesic segment whose length realises $d(x,y)$.

\medskip
We say that $X$ 
is \textit{locally doubling} (LD) if for every positive number $s$ there exists a constant $D_s$ such that
\begin{equation}  \label{f: LDC}
\mu \bigl(2 B\bigr)
\leq D_s \, \mu  \bigl(B\bigr)
\quant B \in \cB_s.
\end{equation}

Given a positive number $\eta$, a set $\fX$ of points in $X$ is a \emph{$\eta$-discretisation} of $X$ if it is maximal with respect to the following property:
$$
\inf\{d(z,w): z,w \in \fX, z \neq w \} >\eta\quad
\hbox{and}\quad d(\fX, x) \leq \eta
\quant x \in X.
$$

It is straightforward to show that $\eta$-discretisations exist for every $\eta$.   By definition the family of balls $\{B_{2\eta}(z)\,:\, z\in \fX\} $ is a covering of $X$. For each subset~$E$ of $X$, we set
$$
\fX_E
:=   \big\{z \in \fX: B_{2\eta}(z) \cap E \neq \emptyset \big\},
$$
and denote by $\sharp \fX_E$ its cardinality.  If $x$ is a point in $X$, we write $\fX_x$ instead of $\fX_{\{x\}}$, for simplicity.
Note that $\sharp\fX_x$ is the number of balls of the covering $\set{B_{2\eta}(z): z \in \fX }$ containing $x$.

We say that the covering $\set{B_{2\eta}(z): z \in \fX }$ has the \textit{finite overlapping property} if
$$
\om
:= \sup_{x \in X} \,\sharp\fX_x
< \infty.
$$

The next lemma contains some known properties of LD spaces.

\begin{lemma}\label{l: balls}
Suppose that $X$ is a LD space.  Assume that $c$ is a positive number and that $\fX$ is a $c/2$-discretisation.  The following hold:
	\begin{enumerate}
		\item[\itemno1]
			the family $\set{B_c(z): z\in \fX}$ is a covering of~$X$ with the finite overlapping property;
		\item[\itemno2]
			for every $R > c$
			$$
			\Upsilon_{c/2,R}^{\fX}
			:= \sup_{B\in \cB^R} \, \sharp \fX_B
			$$
			is finite;
		\item[\itemno3]
			if $X$ is a  geodesic space, then for each positive number~$\de$ there exists a constant~$\ga_\de$ such that
			$$
			\mu\big(B^*) \leq \ga_\de\, \mu(B)
			$$
			for every ball $B$ in $X$ with radius greater than $1$; here $B^*$ denotes the ball with the same centre as $B$ and radius $r_B+\de$.
	\end{enumerate}
\end{lemma}

\begin{proof}
	For the proof of \rmi\ and \rmii\ see \cite[Lemma~1]{MVo}.

We prove \rmiii.  Since $X$ is LD, for every $s$ greater than $1$, the required property holds for all balls with radius at most $s$.
Indeed, there exists $k\in \mathbb{Z}$ such that $(k-1)r_B<\delta\leq kr_B$, and, since $r_B\geq  1$, $k<\log_2 \delta+1$. It follows that
$r_B+\delta\leq 2^{k+1}r_B\leq 2(2^ks)$ and therefore
\[
\mu(B^*)\leq \mu(2^{k+1}B)\leq D_{2^ks}^k\mu(B).
\]
Thus, we can assume that $r_B$ is greater than $10 c$ in the argument below.

	Denote by $B_*$ the ball with the same centre as $B$ and radius $r_B-2c$.  Since  $\fX_{B_*}=\{z\in \fX: d(z,c_B)<r_B-c\}$ and
	the balls $\big\{ B_{c/{4}}(z): z \in \fX_{B_*}\big\}$ are mutually disjoint,
	\begin{equation} \label{f: expansion}	\bigcup_{z\in \fX_{B_*}} \, B_{c}(z)\subset B
		\qquad\hbox{and}\quad
	\sum_{z\in \fX_{B_*}}\, \mu\big(B_{c/{ 4}}(z)\big)
	\leq \mu(B).
	\end{equation}
	Set $\Om := \ds\bigcup_{z\in \fX_{B_*}} \, B_{\de+3c}(z)$.  We \textit{claim} that $B^* \subset \Om$.
 
 Recalling  that $\{ B_c(z): z \in \fX\}$ is a covering of $X$ and that $\fX_{B_*}=\{z\in \fX: B_c(z)\cap B_*\ne \emptyset\}$, we obtain that
\[
B_*\subset \bigcup_{z\in \fX_{B_*}} B_c(z)\subset \Om.
\]

Thus it suffices to show that every point $x$ in $B^*\setminus {B_*}$ is in $\Om$.
	Consider a minimizing geodesic~$\ga$ joining $x$ and~$c_B$.  Denote by $\xi$ the point on $\ga$ at distance
	$r_B-2c$ from $c_B$.  Since $\{B_c(z): z \in \fX\}$ is a covering of $X$ and $\xi$ belongs to $B_*$, there exists a point $z$ in $\fX$ with $d(z,c_B) < r_B-c$
	such that~$\xi$ belongs to $B_c(z)$.  The triangle inequality now implies that
	$$
	d(z,x)
	\leq d(z,\xi) + d(\xi,x)
	< c + r_B +\de - (r_B-2c)
	< 3c+\de.
	$$
	Hence
	$$
	\mu\big(B^*)
	\leq \sum_{z\in \fX_{B_*}}\, \mu\big(B_{\de+3c}(z)\big)
	\leq C\, \sum_{z\in \fX_{B_*}}\, \mu\big(B_{c/{4}}(z)\big)
	\leq C\, \mu (B);
	$$
	the second inequality above follows from the  {LD property} and the third from \eqref{f: expansion}.
\end{proof}

For any positive number $R$, we set
\begin{equation} \label{f: v and V}
	v_X(R)
	:= \inf_{x\in X} \, \mu\big(B_{ R}(x)\big)
	\quad\hbox{and}\quad
	V_X(R)
	:= \sup_{x\in X} \, \mu\big(B_{ R}(x)\big).
\end{equation}
We say that $X$ satisfies the \emph{uniform ball size condition} (UBSC) if
\begin{equation} \label{f: def UBSC}
	0 < v_X(1)
\quad\hbox{and}\quad
	V_X(1) < \infty.
\end{equation}

\begin{lemma}\label{l: UBSC}
	Suppose that $(X,d,\mu)$ is a LD metric measure space satisfying the UBSC \eqref{f: def UBSC}, and denote by $\fX$ a $1$-discretisation of $X$.
	Then for every $R>1$ the following holds
	\begin{equation} \label{f: UBSC}
		V_X(R) \leq D_1\, V_M(1) \, \Upsilon_{1,R}^\fX.
	\end{equation}
\end{lemma}

\begin{proof}
	Suppose that $R$ is greater than $1$, and consider a ball $B$ in~$\cB^R$.
	By Lemma~\ref{l: balls}~\rmi, the family $\{B_2(z): z \in \fX\}$ is a covering of $X$.  Therefore
	$$
	\mu(B)
	\leq \sum_{z\in \fX_B} \, \mu\big(B_{2}(z)\big)
	\leq \sharp\fX_B  \cdot  \sup_{z\in \fX} \, \mu\big(B_{2}(z)\big)
	\leq \sharp\fX_B  \cdot   D_1 \,V_M(1);
	$$
	the last inequality follows from the LD property.  The required conclusion now follows from Lemma~\ref{l: balls}~\rmii.
\end{proof}

\begin{definition} \label{def: RI}
Suppose that $X$ and $X'$ are two metric spaces with distances $d$ and $d'$, respectively.
	A map $\vp : X \to X'$ is a \textit{strict rough isometry} (also referred to as $1$-\textit{quasi-isometries} in the literature) if
\begin{equation} \label{f: RI uno}
	K:= \sup_{x' \in X'}\, d'\big(\vp(X),x'\big)
	<\infty,
\end{equation}
and there exists a nonnegative number $\be$ such that
\begin{equation} \label{f: RI due}
	d(x,y)-\be \le d'\big(\vp(x),\vp(y)\big) \le d(x,y)+\be
	\quant x,y \in X.
\end{equation}
\end{definition}

Points in $X$ and in $X'$ will be customarily denoted by \textit{nonprimed} and \textit{primed} lower case latin letters, respectively.
If $x'$ is a point in $X'$, we denote by $B_r(x')$ and $S_r(x')$ the open ball and the sphere with centre~$x'$ and radius $r$.

\begin{lemma}\label{l: growth}
	Assume that $(X,d,\mu)$ and $(X',d',\mu')$ are LD geodesic spaces enjoying the UBSC \eqref{f: def UBSC}.   Suppose that $\vp$ is a
	strict quasi-isometry between $X$ and $X'$.
	Then there exist positive constants $\Ga_0$, $\Ga_1$ and $R_0$ such that
	\begin{equation}\label{f:l:growth}
			\Ga_0\, \mu\big(B_R(x)\big) \leq \mu'\big(B_R(x')\big) \leq \Ga_1\, \mu\big(B_R(x)\big)
\end{equation}
	for every $x'$ in $\vp(X)$, for every $x$ in $\vp^{-1}\big(\{x'\}\big)$, and for every $R\geq R_0$.
\end{lemma}

\begin{proof}
	Denote by $K$ and $\be$ the parameters of $\vp$ (see Definition~\ref{def: RI}).  Choose
	$$
	\kappa = \max \big(K,\be+2)
	\quad\hbox{and}\quad
	R_0
	:= 6\,(2\kappa+\be)+1.
	$$
	Denote by~$\fX'$ a maximal (with respect to inclusion) set of points in $\vp(X)$ such that
	\begin{equation} \label{f: separation fMprimed}
	d'(x',y') \geq \kappa
	\quant x',y'\in\fX'.
	\end{equation}
	Then the family of balls $\big\{B_{2\kappa}(z'): z' \in \fX' \big\}$ is a covering of $X'$.  Indeed, if $v'$ is a point in $X'$, then there exist
	$x'$ in $\vp(X)$ such that $d'(v',x') \leq K$, and $z'$ in $\fX'$ at distance at most $\kappa$ from $x'$.  The triangle inequality now implies that
	$$
	d'(v',z')
	\leq K+\kappa
	< 2\kappa.
	$$
	Furthermore, the balls $\big\{B_{\kappa/2}(z'): z' \in \fX' \big\}$ are mutually disjoint.

	For each $z'$ in $\fX'$, choose (arbitrarily) a point $z$ in $\vp^{-1}\big(\{z'\}\big)$, and denote by~$\fX$ the collection of
	such points.  By construction, the restriction of~$\vp$ to~$\fX$ is a bijective map between $\fX$ and $\fX'$.

	Observe that the points in $\fX$ are $(\kappa -\be)$-separated.  Indeed, suppose that~$z$ and~$z_1$ are in $\fX$.  Then
	$$
	d(z,z_1)
	\geq d\big(\vp(z),\vp(z_1)\big) - \be
	\geq \kappa-\be;
	$$
	the last inequality follows from \eqref{f: separation fMprimed}.  Therefore the balls in the family
	$$
	\big\{B_{(\kappa-\be)/2}({z }): z \in \fX \big\}
	$$
	are mutually disjoint.  Furthermore the collection of balls
	$$
	\big\{B_{2\kappa+\be}(z): z \in \fX \big\}
	$$
	is a covering of~$X$.  Indeed, suppose that $x$ is in $X$.  Since $\big\{B_{2\kappa}(z'): z' \in \fX' \big\}$ is a covering of~$X'$, there exists a point
	$z'$ in $\fX'$ such that $d\big(\vp(x),z'\big) < 2\kappa$.  Denote by $z$ the point in $\fX$ such that $\vp(z) = z'$.  Then
	$$
	d(x,z)
	\leq d\big(\vp(x), \vp(z)\big) + \be
	\leq 2\kappa + \be,
	$$
	as required.

	Now, suppose that $R\geq R_0$, and consider the ball $B_R(x')$, where $x'$ is any point in {{$\vp(X)$}}.
	Since $\big\{B_{2\kappa}(z'): z' \in \fX' \big\}$ is a covering of $X'$, there exists a point~$z'$ in $\fX'$ such that $d'(x',z') < 2\kappa$.  Therefore
	$B_R(x') \supseteq B_{R-2\kappa}(z')$, so that
	\begin{equation} \label{f: II}
	\begin{aligned}
	\mu'\big(B_R(x')\big) 	
	& \geq \mu'\big(B_{R-2\kappa}(z')\big) \\	
	& \geq \sum_{z'_1\in (\fX' \cap B_{R-{ 5\kappa/2}}(z'))} \, \mu'\big(B_{\kappa/2} (z'_1)\big) \\
	& \geq v_{X'}(\kappa/2)  \cdot \sharp \big(\fX' \cap B_{R-{ 5\kappa/2}}(z')\big);
	\end{aligned}
	\end{equation}
	the last inequality follows from the disjointness of the balls in $\{B_{\kappa/2}(z'_1): z'_1 \in \fX' \big\}$.

	Denote by $z$ the unique point in $\fX$ such that $\vp(z) = z'$.  Observe that if $z_1$ is a point in $\fX$ such that $d(z_1,z)<R-{ 5\kappa/2}-\be$, then
	$$
	d'\big(\vp(z_1),z'\big)
	\leq d\big(z_1,z\big) + \be
	<R-{ 5\kappa/2},
	$$
	so that
	\begin{equation} \label{f: III}
	\sharp \big(\fX' \cap B_{R-{ 5\kappa/2}}(z')\big)
	\geq \sharp \big(\fX \cap B_{R-{ 5\kappa/2}-\beta}(z)\big).
	\end{equation}
	The triangle inequality implies that if $z_1$ is in $\fX$ and $B_{2\kappa+ \be} (z_1)$ has nonempty intersection with $B_{R-{9\kappa/2}-2\beta}(z)$,
	then $d(z,z_1) < R-{5\kappa/2}-\be$.  Since $\big\{B_{2\kappa+\be}(z_1): z_1 \in \fX \big\}$ is a covering of $X$, we conclude that
	$$
	\begin{aligned}
	\mu\big(B_{R-4\kappa-2\be}(z)\big)
		& \leq \sum_{z_1\in (\fX \cap B_{R-{5\kappa/2}-\beta}(z))}\, \mu\big(B_{2\kappa+\be}(z_1)\big) \\
		& \leq V_X(2\kappa+\be) \cdot \sharp\big(\fX \cap B_{R-{5\kappa/2}-\beta}(z)\big).
	\end{aligned}
	$$
	This inequality, \eqref{f: II} and \eqref{f: III} imply that
	$$
	\mu\big(B_{R-{9\kappa/2}-2\be}(z)\big)
	\leq \frac{V_X(2\kappa+\be)}{v_{X'}(\kappa/2)} \, \mu'\big(B_R(x')\big).
	$$
	Now,
    choose any point $x$ in $\vp^{-1}\big(\{x'\}\big)$.  Observe that
	$$
	d(x,z)
	\leq d\big(\vp(x),\vp(z)\big) + \be
	\leq 2\kappa + \be,
	$$
	so that the ball $B_{R- 13\kappa/2-3\be}(x)$ is contained in $B_{R-{9\kappa/2}-2\be}(z)$.  Consequently
	$$
	\mu\big(R- {13\kappa/2}-3\be(x)\big)
	\leq \frac{V_X(2\kappa+\be)}{v_{X'}(\kappa/2)} \,  \mu'\big(B_R(x')\big).   	
	$$
	By Lemma~\ref{l: balls}~\rmiii,
	$$
	\mu\big(B_{R}(x)\big)
	\leq \ga_{{13\kappa/2}+3\be}\, \mu\big(B_{R-{13\kappa/2}-3\be}(x)\big).
	$$
	Combining this with the penultimate inequality proves the left inequality in \eqref{f:l:growth}, with
	$$
	\Ga_0^{-1} = \ga_{{13\kappa/2}+3\be}\, \frac{V_X(2\kappa+\be)}{v_{X'}(\kappa/2)}.
	$$

	The proof of the right inequality in \eqref{f:l:growth} is similar, and it is omitted.
\end{proof}

The next theorem is the main result of this section.  Preliminarily, we need more notation.  Suppose that $(X,d,\mu)$ and $(X',d',\mu')$ and $\vp$ are as in
the statement of the lemma above. 
For any positive number $\si$ and any function $f$ on $X'$ denote by~$\pi_\si f$ the function on $X'$, defined by
\begin{equation} \label{f: pi f}
	(\pi_\si f) (x')
	:= \int_{B_\si(x')} \mod{f} \wrt \mu'
	\quant x'\in X'.
\end{equation}
Similarly, if $G$ is a function on $X$, define $\pi_\si G$ by
\begin{equation} \label{f: pi f}
	(\pi_\si G) (x)
	:= \int_{B_\si(x)} \mod{G} \wrt \mu
	\quant x\in X.
\end{equation}
Given a function $f$ on $X'$, and a positive number $R$, consider the averages 
$$
	A_R f(z')
	:= \frac{1}{\mu'\big(B_R(z')\big)} \, \int_{B_R(z')} \mod{f} \wrt \mu'
	\quant z' \in  X',
$$
and
$$
	\cA_R f(z')
	:= \sup_{B\in \cB^R: B \ni z'} \, \frac{1}{\bigmod{B}} \, \int_{B} \mod{f} \wrt \mu'
	\quant z' \in  X'.
$$
For any number $\si$ greater than $1$, define the maximal operators
\begin{equation} \label{f: MK I}
	\cM^\si\! f (z')
	:= \sup_{1< R\leq \si} \, A_R \mod{f}(z')
	\quant z' \in X'
\end{equation}
and
\begin{equation} \label{f: MK II}
	\cM_\si f(z')
	:= \sup_{R>\si} \, A_R \mod{f}(z')
	\quant z' \in X'.
\end{equation}

\smallskip
\begin{theorem} \label{t: quasi iso}
	Assume that $(X,d,\mu)$ and $(X',d',\mu')$ are LD  {geodesic} spaces enjoying the UBSC \eqref{f: def UBSC}.   Suppose that $\vp$ is a
	strict quasi-isometry between $X$ and $X'$, that $1\leq p<\infty$, and that $1\leq r\leq s\leq \infty$.  The following hold:
\begin{enumerate}
	\item[\itemno1]
		if $\cM_\infty$ is bounded from $\lorentz{p}{r}{X}$ to $\lorentz{p}{s}{X}$, then $\cM_\infty$ is bounded from $\lorentz{p}{r}{X'}$
		to $\lorentz{p}{s}{X'}$;
	\item[\itemno2]
		if $\cN_\infty$ is bounded from $\lorentz{p}{r}{X}$ to $\lorentz{p}{s}{X}$, then $\cN_\infty$ is bounded from $\lorentz{p}{r}{X'}$
		to $\lorentz{p}{s}{X'}$.
\end{enumerate}
\end{theorem}

\begin{proof}
	First we prove \rmi.
	Set $\kappa := \max(K,\be+2)$, where $K$ and $\be$ are as in Definition~\ref{def: RI}, and recall the constant $R_0$ defined at
	the beginning of the proof of Lemma~\ref{l: growth}.
	Clearly $\cM_\infty f \leq \cM^{R_0}\! f + \cM_{R_0} f$. Thus, it suffices to show that both~$\cM^{R_0}$ and $\cM_{R_0}$ satisfy the required norm estimates.

	\textit{Step I: estimate of $\cM^{R_0}$}.
	Suppose that $n$ is an integer in $[0,{R_0})$ and that~$R$ belongs to the interval $[n,n+1)$.  Then the LD property of $X'$ implies that
	$$
	\begin{aligned}
	A_R \mod{f}(z')
		=    \frac{\mu'\big(B_{n+1}(z')\big)}{\mu'\big(B_R (z')\big)} \, \frac{1}{\mu'\big(B_{n+1}(z')\big)} \, \int_{B_R(z')} \mod{f} \wrt \mu'
		\leq D_{R_0} \, A_{n+1} \mod{f}(z').
	\end{aligned}
	$$
	Therefore 
		$$
	\cM^{R_0} f(z')
	\leq \max_{1\leq n\leq {R_0}} \sup_{n\leq R< n+1} \, A_R \mod{f}(z')
	\leq D_{R_0}\, \sum_{n\leq {R_0}} \, A_{n+1} \mod{f}(z').
	$$
	Consider the set
	$$
	E_R:= \{(y',z') \in X'\times X': d'(y',z') < R \}.
	$$
	Observe that
	$$
	\begin{aligned}
		\bignorm{A_{n+1} \mod{f}}{\lu{X'}}
		& = \int_{X'}\frac{\!\!\wrt \mu'(z')}{\mu'\big(B_{n+1}(z')\big)} \,\int_{X'} \One_{E_{n+1}}(y',z') \, \mod{f(y')} \wrt\mu' (y') \\
		& = \int_{X'}\!\!\wrt \mu'(y') \, \mod{f(y')}\,\int_{B_{n+1}(y')} \frac{\wrt\mu' (z')}{\mu'\big(B_{n+1}(z')\big)}  \\
		& \leq \frac{V_{X'}(n+1)}{v_{X'}(n+1)}\, \bignorm{f}{\lu{\mu'}}.
	\end{aligned}
	$$
	As a consequence
	$$
	\begin{aligned}
		\bignorm{\cM^{R_0}\! f}{\lu{X'}}
		\leq {R_0}\, D_{R_0}\, \frac{V_{X'}(R_0+1)}{v_{X'}(1)}\, \bignorm{f}{\lu{X'}}
		\quant f \in \lu{X'}.
	\end{aligned}
	$$
	Trivially $\bignorm{\cM^{R_0}\! f}{\ly{X'}} \leq \bignorm{f}{\ly{X'}}$.
	By interpolation, it follows that $\cM^{R_0}$ is bounded on the Lorentz space $\lorentz{p}{r}{X'}$ whenever $1< p<\infty$ and $1\leq r\leq \infty$.

	A variant of the argument above yields the boundedness of {$\cM^{R_0}$} on the Lorentz space $\lorentz{1}{r}{X'}$ whenever $1< r\leq \infty$.
	We omit the details.

		\medskip
	\textit{Step II: the operator $\La$ and maximal operators}.  Let $\fX$ and $\fX'$ be as in the proof of Lemma~\ref{l: growth}.  Recall that
	$\{B_{2\kappa}(m'): m'\in \fX'\}$ is a covering of $X'$, and $\{B_{2\kappa+\be}(m):  m \in \fX\}$ is a covering of~$X$.

	Consider a ball $B_R(x')$ in $X'$, with $R\geq R_0$, and denote by ${\fX}_{R,x'}'$ the set of all~$z'$ in $\fX'$ such that $B_{2\kappa}(z')$ has nonempty
	intersection with $B_R(x')$.  Then
	$\ds
	\bigcup_{z'\in {\fX}_{R,x'}'} \, B_{2\kappa}(z')
	$
	covers $B_R(x')$, so that
	$$
	\begin{aligned}
	\int_{B_R(x')} \, f \wrt \mu'
		\leq \sum_{z'\in{\fX}_{R,x'}'} \, \int_{B_{2\kappa}(z')} \, f \wrt\mu'
		= \sum_{z'\in {\fX}_{R,x'}'} \, \pi_{2\kappa} f(z').
	\end{aligned}
	$$
	If $x'$ belongs to $\vp(X)$, then consider any point $x$ in $\vp^{-1}\big(\{x'\}\big)$.  Denote by $\fX_{R,x}$ the set of all points $z$ in $\fX$
	such that $B_{2\kappa+\be}(z)$ intersects the ball $B_R(x)$.  Note that if $z'$ belongs to $\fX_{R,x'}$, then $d'(z',x') < R+2\kappa$, so that
	$$
	d\big(\vp^{-1} (z'),x\big)
	\leq d'(z',x') + \be
	< R+2\kappa + \be.
	$$
Therefore, if $z'$ belongs to $\fX_{R,x'}$, then    {$(\vp_{|\fX})^{-1}$} (which is a single point in $\fX$) belongs to $\fX_{R,x}$.  Therefore
	\begin{equation} \label{f: int BRprime I}
	\int_{B_R(x')} \, f \wrt \mu'
	\leq \sum_{z\in \fX_{R,x}} \, F(z),
	\end{equation}
	where we have set
	\begin{equation} \label{f: F}
		F := (\pi_{2\kappa} f)\circ \vp.
	\end{equation}
	Consider the function
	$$
	\Psi
	:= \sum_{z\in\fX} \, \One_{B_{2\kappa+\be}(z)}.
	$$
	Observe that
	$
	1\leq \Psi \leq \om,
	$
	where $\om$ denotes the overlapping number of the covering $\{B_{2\kappa+\be}(z):  z \in \fX\}$ of~$X$.  Then for every $z$ in $\fX$ set
	$$
	\psi_z
	:= \Psi^{-1}\cdot \One_{B_{2\kappa+\be}(z)}.
	$$
	Define the function $\cE F$ on $X$ by
	\begin{equation} \label{f: cE F}
	\cE F
	:= \om\, \sum_{\ell\in \fX} \, F(\ell) \, \psi_\ell
	\quant x \in X.
	\end{equation}
	Since $1/\om \leq \psi_z \leq 1$ on $B_{2\kappa+\be}(z)$, the inequality $\cE F(z) \geq F(z)$ holds for every~$z$ in $\fX$.
	Then
	$$
	\begin{aligned}
		v_X(2\kappa+\be) \cdot F(z)
		& \leq \mu\big(B_{2\kappa+\be}(z)\big) \, F(z) \\
		& \leq \om\, \int_{B_{2\kappa+\be}(z)} \, \psi_z \wrt \mu \cdot F(z) \\
		& \leq \om\, \int_{B_{2\kappa+\be}(z)} \, \cE F \wrt \mu.
	\end{aligned}
	$$
	This estimate and \eqref{f: int BRprime I} imply that
	$$
	\int_{B_R(x')} \, f \wrt \mu'
	\leq \frac{\om}{v_X(2\kappa+\be)}\, \sum_{z\in \fX_{R,x}} \, \int_{B_{2\kappa+\be}(z)} \, \cE F \wrt \mu.
	$$
	Furthermore
	$$
	\sum_{z\in \fX_{R,x}} \, \int_{B_{2\kappa+\be}(z)} \, \cE F \wrt \mu
	\leq \om\, \int_{B_{R+2(2\kappa+\be)}(x)} \, \cE F \wrt \mu,
	$$
	because $B_{2\kappa+\be}(z)$ is contained in $B_{R+2(2\kappa+\be)}(x)$ for every $z$ in $\fX_{R,x}$.
	Altogether we have proved that for $R$ large and for each $x'$ in  ${\vp(X)}$ and for each $x$ in $\vp^{-1}\big(\{x'\}\big)$
	\begin{equation} \label{f: comparison integrals}
		\int_{B_R(x')} \, f \wrt \mu'
		\leq \frac{\om^2}{v_X(2\kappa+\be)}\, \int_{B_{R+2(2\kappa+\be)}(x)} \, \cE F \wrt \mu.
	\end{equation}
	Observe that
	$$
	\begin{aligned}
		\frac{\mu\big(B_{R+2(2\kappa+\be)}(x)\big)}{\mu'\big(B_R(x')\big)}
		 = 
		    \frac{\mu\big(B_{R+2(2\kappa+\be)}(x)\big)}{\mu\big(B_R(x)\big)}
		    \cdot \frac{\mu\big(B_{R}(x)\big)}{\mu'\big(B_R(x')\big)}
		 \leq \ga_{4\kappa+2\be} \, \Ga_0^{-1},
	\end{aligned}
	$$
	by Lemma~\ref{l: balls}~\rmiii\ (where the constant {$\ga_{4\kappa+2\be}$} is defined) and Lemma~\ref{l: growth}\
	(where the constant $\Ga_0$ is defined).  This, and \eqref{f: comparison integrals} imply that
	\begin{equation} \label{f: comparison integrals II}
		\cM_{{R_0+2\kappa}} f (x')
		\leq  \ga_{{4\kappa}+2\be} \, \Ga_0^{-1} \,\frac{\om^2}{v_X(2\kappa+\be)}\, \cM_{{R_0+6\kappa+2\beta}} (\cE F) (x)
\end{equation}
	for each $x'$ in $\vp(X)$, and every~$x$ in $\vp^{-1}\big(\{x'\}\big)$. 
	Consider the linear operator $\La$, acting on $\lu{X'}+\ly{X'}$, defined by
	$$
	\La f
	:= \cE F;
	$$
	$F$ and $\cE F$ are defined in \eqref{f: F} and \eqref{f: cE F}, respectively.  Clearly $\La$ maps functions on $X'$ to functions on $X$.

	We \textit{claim} that $\La$ is bounded from $\lorentz{p}{r}{X'}$ to $\lorentz{p}{r}{X}$ whenever $1< p<\infty$ and $1\leq r\leq \infty$.
	
	Indeed, $\La$ is bounded from $\ly{X'}$ to $\ly{X}$, for
	$$
	\bignorm{\cE F}{\ly{X}}
	\leq {\om^2}\, \bignorm{F}{\ly{\fX}}
	=    {\om^2}\, \bignorm{\pi_{ 2\kappa} f}{\ly{\fX'}}
	=    {\om^2\, V_{X'}({2\kappa})}\bignorm{f}{\ly{X'}}.
	$$
	Furthermore, $\La$ is bounded from $\lu{X'}$ to $\lu{X}$.  Indeed, notice that
	$$
	\begin{aligned}
	\bignorm{\cE F}{\lu{X}}
		& \leq \om \, V_X(2\kappa+\be) \, \sum_{z\in\fX} \, \mod{F(z)} \\
		& =    \om \, V_X(2\kappa+\be) \, \bignorm{F}{\lu{\fX}} \\
		& \leq \om \, V_X(2\kappa+\be) \, \bignorm{\pi_{2\kappa} \mod{f}}{\lu{\fX'}}.
	\end{aligned}
	$$
	The latter norm is equal to
	$$
	\begin{aligned}
		\sum_{z'\in\fX'} \, (\pi_{2\kappa} \mod{f}) (z')
		& \leq   \int_{X'} \Big(\sum_{z'\in\fX'} \, \One_{B_{2\kappa}}(z')\Big) \mod{f} \wrt\mu' \\
		& \leq  \om' \, \bignorm{f}{\lu{X'}},
	\end{aligned}
	$$
where $\om'$ denotes the overlapping number of the covering $\{B_{2\kappa}(z'): z'\in \fX'\}$ (see Lemma~\ref{l: balls} \rmi).  Thus,
	$$
	\bignorm{\cE F}{\lu{X}}
	\leq \om \, V_X(2\kappa+\be) \, \om' \, \bignorm{f}{\lu{X'}},
	$$
	as required.
	
	Now the claim follows from Marcinkiewicz's interpolation theorem.
	
\medskip
\textit{Step III: conclusion of the proof of \rmi}.
Suppose that $\al>0$ and that $x'$ is a point in $E_{\cM_\kappa f}(\al)$.  Denote by $\vr(x')$ one of the points in $\fX'$ closest to $x'$, and observe that
$d'(x',\vr(x')) < 2\kappa$, because $\{B_{2\kappa} (p'): p' \in \fX'\}$ is a covering of $X'$.

The triangle inequality implies that the ball $B_R(x')$ is contained in $B_{R+2k}(\vr(x'))$, so that
$$
\frac{1}{\mu'\big(B_R(x')\big)}\, \int_{B_R(x')} \! f \wrt \mu'
\leq \frac{\mu'\big(B_{R+2\kappa}(\vr(x'))\big)}{\mu'\big(B_R(x')\big)}\, \frac{1}{\mu'\big(B_{R+2\kappa}(\vr(x'))\big)}\, \int_{B_{R+2\kappa}(\vr(x'))}\!  f \wrt \mu'.
$$
Since $B_{R+2\kappa}(\vr(x'))$ is, in turn, contained in $B_{R+4\kappa}(x')$, Lemma~\ref{l: balls}~\rmiii\ implies that
$$
\frac{\mu'\big(B_{R+2\kappa}(\vr(x'))\big)}{\mu'\big(B_R(x')\big)}
\leq \frac{\mu'\big(B_{R+4\kappa}(x')\big)}{\mu'\big(B_R(x')\big)}
\leq \ga_{4\kappa}.
$$
Consequently
$$
\cM_{R_0} f(x') \leq \ga_{4\kappa} \, \cM_{{R_0+2\kappa}} f(\vr(x')).
$$
Given a point $z'$ in $\fX'$, the set of points $x'$ in $X'$ such that $\vr(x') =z'$ is contained in the ball $B_{2\kappa}(z')$, whence
$$
\mu'\big(E_{\cM_{R_0} f}(\al)\big)
\leq V_{X'}(2\kappa) \cdot  \sharp\big(\big\{z'\in \fX': \cM_{ {R_0+2\kappa}} f(z')>\al/\ga_{4\kappa}\big\}\big).
$$
Now we use \eqref{f: comparison integrals II}, and obtain that
$$
\mu'\big(E_{{\cM_{R_0}}}(\al)\big)
\leq V_{X'}(2\kappa) \cdot  \sharp\big(\big\{z\in \fX: \cM_{{R_0+6\kappa+2\beta}} (\cE F)(z)>\al/\si\big\}\big),
$$
where $\ds\si =\ga_{{4\kappa+2\beta}} \, \Ga_0^{-1} \,\frac{\om^2}{v_X(2\kappa+\be)}\,\ga_{4\kappa}$.
Observe that if $\cM_{{R_0+6\kappa+2\beta}} (\cE F) (z) > \al/\si$, and~$x$ is in $B_{2\kappa+\be}(z)$, then $\cM_{{R_0+8\kappa+3\beta}} (\cE F) (x) > C\al/\si$,
where 
\[
C:=\inf_{x\in X} \frac{\mu(B_{R_0+4\kappa+\beta}(x))}{\mu(B_{R_0+8\kappa+3\beta}(x))}>0
\] 
by the LD assumption. Therefore
$$
\begin{aligned}
	& \sharp\big(\big\{z\in \fX: \cM_{{R_0+6\kappa+2\beta}} (\cE F)(z)>\al/\si\big\}\big) \\
	& \leq \frac{1}{v_X(2\kappa+\be)} \, \mu\big(\big\{x\in X: \cM_{ {R_0+8\kappa+3\beta}} (\cE F)(x)>C \,\al/\si\big\}\big).
\end{aligned}
$$
Altogether we obtain that
$$
\mu'\big(E_{\cM_{R_0} f}(\al)\big)
\leq \frac{V_{X'}(2\kappa)}{v_X(2\kappa+\be)} \, \mu\big(E_{\cM_{{R_0+8\kappa+3\beta}} (\cE F)}(C \,\al/\si)\big).
$$
As a consequence
$$
\bignorm{\cM_{R_0} f}{\lorentz{p}{s}{X'}}
\leq \Big(\frac{V_{X'}(2\kappa)}{v_X(2\kappa+\be)} \Big)^{1/p} \Big(p\ioty  \mu\big(E_{\cM_{{R_0+8\kappa +3\beta}}(\cE F)}(C\al/\si)\big)^{s/p} \,
    \al^{s-1} \wrt \al\Big)^{1/s}.\,
$$
Changing variables and using the definition of the (quasi-) norm of $\lorentz{p}{s}{X'}$, we see that
$$
\bignorm{\cM_{R_0} f}{\lorentz{p}{s}{X'}}
\leq (\si/C)\, \Big(\frac{V_{X'}(2\kappa)}{v_X(2\kappa+\be)} \Big)^{1/p}\, \bignorm{\cM_{{R_0+8\kappa+ 3\beta}} \La f}{\lorentz{p}{s}{X}}.
$$
Since, by assumption, $\cM$, hence $\cM_{{R_0+8\kappa+3\beta}}$, is bounded from $\lorentz{p}{r}{X}$ to $\lorentz{p}{s}{X}$, and $\La$
is bounded from $\lorentz{p}{s}{X}$ to $\lorentz{p}{s}{X'}$,
$$
\bignorm{\cM_{R_0} \La f}{\lorentz{p}{s}{X}}
\leq \bigopnormto{\cM}{L^{p,r}(X);L^{p,s}(X)} \, \bignorm{\La f}{\lorentz{p}{r}{X}}
\leq C\, \bignorm{f}{\lorentz{p}{r}{X'}},
$$
as required to conclude the proof of Step~III, and of \rmi.

\smallskip
The proof of \rmii\ copies the strategy of the proof of \rmi\ with a few slight variants; we omit the details.
\end{proof}

\section{Connected sums}  \label{s: Connected}
One of the themes of this research is to understand to what extent the geometry at infinity of the manifold under consideration affects the
boundedness properties of $\cM$ and $\cN$.  The state of our knowledge is somehow embryonic, and we can only discuss a few, hopefully paradigmatic, examples.

In this section we show that by suitably gluing together two space forms with curvatures $-a^2$ and $-b^2$ the $L^p$-boundedness properties of $\cN$
change abruptly, whereas those of $\cM$ do not.  This is only partially surprising, given the fact that even on space forms the $L^p$-boundedness
properties of~$\cM$ and~$\cN$ are significantly different.

In particular, we consider the connected sum $M := M_a\sharp M_b$ of two simply connected
$m$-dimensional space forms $M_a$ and $M_b$ with curvatures $-a^2$ and $-b^2$, respectively.
The topology of $M$ is nontrivial, for $M$ has two ends.  The gluing of $M_a$ and $M_b$ can be taylored so that
\begin{equation} \label{f: dec M}
	M = \Om_a \cup K \cup \Om_b,
\end{equation}
where $K$ is a compact subset of $M$ with $\diam(K) = 3$, $\Om_a$ and $\Om_b$ are open subsets of $M$ isometric to $M_a\setminus \OV{B_1^a(o_a)}$
and $M_b\setminus\OV{B_1^b(o_b)}$, respectively, and the union above is disjoint.  We refer to $\Om_a$ and to $\Om_b$ as to the upper and
the lower leaf of $M$, respectively.

It is useful to consider also the subset $\Om_a'$ of $\Om_a$ isometric to $M_a\setminus \OV{B_3^a(o_a)}$
and the subset $\Om_b'$ of $\Om_b$ isometric to $M_a\setminus \OV{B_3^b(o_b)}$ and the compact set
$$
K'
:= M\setminus \big(\Om_a'\cup\Om_b'\big).
$$
Notice that $K' = \{x \in M: d(x, K) \leq 2\}$.

Denote by $M_\kappa$ the $m$ dimensional space form with curvature $-\kappa^2$ and Riemannian measure $\mu_\kappa$.  Set $q:=\e^{m-1}$. 
It is well known that there are constants~$\eta_1$ and~$\eta_2$, depending on $m$ and $\kappa$ such that
\begin{equation} \label{f: volume kappa}
	\eta_{1,\kappa}\, q^{\kappa R} 
	\leq \mu_\kappa\big(B_R(x)\big)
	\leq \eta_{2,\kappa}\, q^{\kappa R} 
	\quant R \ge 1 \quant x\in M_\kappa.
\end{equation}
Note also that there exists a constant $\eta_\kappa$ such that 
\begin{equation} \label{f: volume kappa I}
\mu_\kappa\big(B_R(x)\big) \sim \eta_b\, q^{\kappa R}
\quad\hbox{as $R$ tends to infinity}.
\end{equation}

\begin{theorem} \label{t: connected sum}
	The following hold:
	\begin{enumerate}
		\item[\itemno1]
			the maximal operator $\cM_\infty$ is of weak type $(1,1)$ on $M$;
		\item[\itemno2]
			the maximal operator $\cN_\infty$ is bounded on $\lp{M}$ if and only if $p=\infty$.
	\end{enumerate}
\end{theorem}

\begin{proof}
	In order to simplify the proof we shall assume that the metric $g$ on $M$ has the form $\wrt t^2 + \si(t)^2 \, g_{\BS^{n-1}}$, where 
	\begin{equation}\label{eq:model}
	\si(t)
	= \begin{cases}
	a^{-1}\, \sinh (at) & \hbox{if $t>1$} \\
	b^{-1}\, \sinh (-bt) & \hbox{if $t<-1$}, 
	\end{cases}
	\end{equation}
	and $\si$ it is strictly positive on $[-1,1]$.
	However, since compact perturbations of the metric do not alter our conclusions, because of Theorem~\ref{t: quasi iso}, the results remains true with different, 
	possibly non-rotationally symmetric, gluings.

	For the duration of this proof, for every measurable subset $E$ of $M$ we shall write $\mod{E}$ instead of $\mu(E)$, for brevity.
	We also write $\bignorm{g}{1}$ instead of $\bignorm{g}{\lu{M}}$.

	First we prove \rmi.  It is not hard to check that it suffices to prove the desired weak type $(1,1)$ estimates for the maximal operator
	$$
	\cM_\infty' f(x)
	:= \sup_{r> 7} \, \frac{1}{\bigmod{B_r(x)}} \, \int_{B_r(x)} \, \mod{f} \wrt \mu
	$$
	(the supremum is taken only over balls of radius greater than $7$, whereas in the definition of $\cM_\infty$ the
	supremum is taken over balls of radius greater than $1$).
	Suppose that $f$ is in $\lu{M}$; we may assume that $f$ is nonnegative.  Write
	$$
	f = f_a + f_K + f_b,
	$$ \
	where
	$$
	f_a = f\, \One_{\Om_a}, \quad
	f_b = f\, \One_{\Om_b} \quad\hbox{and}\quad
	f_K = f\, \One_{K}.
	$$
	Since
	$$
	\big\{\cM_\infty f> \la\big\}
	\subseteq \big\{\cM_\infty f_a> \la/3\big\}
	\cup \big\{\cM_\infty f_K> \la/3\big\}
	\cup \big\{\cM_\infty f_b> \la/3\big\},
	$$
	it suffices to prove that there exists a constant $C$ such that
	\begin{equation} \label{f: estimate fa}
		\bigmod{\big\{\cM_\infty f_a> \la/3\big\}}
	\leq \frac{C}{\la} \, \bignorm{f_a}{1}
	\quant \la >0,
	\end{equation}
	and that similar estimates hold with $f_K$ and $f_b$ in place of $f_a$.

	Denote by $\vp_a$ the isometry between $M_a\setminus \OV{B_1^a(o_a)}$ and $\Om_a$, which was alluded to at the beginning of this section.  Note that
	$$
	\bignorm{f_a \circ \vp_a}{\lu{M_a}}
	= \bignorm{f_a}{1},
	$$
	because the support of $f_a$ is contained in $\Om_a$ and $\mu = (\vp_a)^*\mu$.
	
	\smallskip
	In order to prove \eqref{f: estimate fa}, we estimate $\cM_\infty f_a(x)$ by considering the cases where
	$x$ belongs to $\Om_a'$, to $K'$ and to $\Om_b'$ separately.

	\smallskip
	\textit{Suppose first} that $x$ belongs to $\Om_a'$.  We need to estimate $\mu\big(B_R(x)\big)$ from below.
	Assume that $R\geq 5$, and set
	$$
	E_R (x) := \vp_a \big(B_{R-3}^a(\vp_a^{-1}(x)) \setminus B_1^a(o_a)\big).
	$$
	We \textit{claim} that $B_R(x) \cap \Om_a$ contains~$E_R$.

	Indeed, suppose that $y$ belongs to $E_R$ (hence to $\Om_a$), and consider the geodesic~$\ga$ joining $\vp_a^{-1}(x)$ and $\vp_a^{-1}(y)$ in $M_a$.

	If $\ga$ does not intersect $B_1(o_a)$, then $\vp_a(\ga)$ does not intersect $K$, so that
	$$
	d(x,y)
	\leq \ell\big(\vp_a(\ga)\big)
	=    \ell_a(\ga)
	<    R-3.
	$$
	If $\ga$ intersects $B_1^a(o_a)$, denote by $\xi$ and $\eta$ the (possibly overlapping) points in $S_1(o_a)$ that belong to $\ga$.
	Then $\vp_a(\xi)$ and $\vp_a(\eta)$ are points in $K$, so that
	$$
	d\big(\vp_a(\xi),\vp_a(\eta)\big)
	\leq \diam\, K
	\leq 3.
	$$
	Furthermore 
	$$
	d_a\big(\vp_a^{-1}(x), \xi) + d_a\big(\eta,\vp_a^{-1}(y)\big)
	= d_a \big(\vp_a^{-1}(x), \vp_a^{-1}(y)\big) - d_a(\xi,\eta)
	< R-3.
	$$
	Therefore
	$$
	\begin{aligned}
	d(x,y)
		& \leq d\big(x,\vp_a(\xi)\big) + d\big(\vp_a(\xi),\vp_a(\eta)\big) + d\big(\vp_a(\eta),y\big) \\
		& \leq d_a\big(\vp_a^{-1}(x),\xi\big) + \diam\, K + d_a\big(\eta,\vp_a^{-1}(y)\big) \\
		& < R,
	\end{aligned}
	$$
	whence $y$ is in $B_R(x)$, and the claim is proved.
	Since, by Str\"omberg's result,  $\cM_{\infty,a}$ (the maximal function on $M_a$) is of weak type $(1,1)$, there exists a constant $C$ such that
	$$
	\mu_a\big(\big\{y\in \vp_a^{-1} (\Om_a'): \cM_{\infty,a} (f_a\circ \vp_a) (y) > \la/(3C)\big\}\big)
	\leq  \frac{C}{\la}\, \bignorm{f_a\circ \vp_a}{\lu{M_a}}.
	$$
	Finally, observe that $\bignorm{f_a\circ \vp_a}{\lu{M_a}} = \bignorm{f_a}{1}$, so that
	\begin{equation} \label{f: fa I}
		\bigmod{\{x\in \Om_a': \cM_{\infty}f_a (x) > \la/3\}}
	\leq \frac{C}{\la}\, \bignorm{f_a}{1}
	\quant \la > 0.
	\end{equation}

	\medskip
	Next, \textit{suppose that} $x$ is in $K'$.  Since the sectional curvature of $M$ is bounded from above and the injectivity radius is positive, 
	$M$ possesses the uniform ball size condition (see \eqref{f: def UBSC}), so that
	\begin{equation} \label{f: constant al}
	\al := 
		\inf_{x\in M} \, \mu\big(B_1(x)\big)
	>0.
	\end{equation}
	Thus, trivially
	$$
	\frac{1}{\mu\big(B_R(x)\big)} \, \int_{B_R(x)} f_a \wrt\mu
	\leq \al^{-1}\, \bignorm{f_a}{1}
	\quant x \in K'.
	$$
	Hence
	$$
	\bigmod{\big\{x\in K': \cM_\infty f_a (x) > \la/3\big\}}
	\leq
	\begin{cases}
		\mod{K'} & \hbox{if $\la<(3/\al)\bignorm{f_a}{1}$} \\
		0    	& \hbox{otherwise},
	\end{cases}
	$$
	so that
	\begin{equation} \label{f: fa II}
		\bigmod{\big\{x\in K': \cM_\infty f_a (x) > \la/3\big\}}
		\leq \frac{3}{\al}\, \frac{\mod{K'}}{\la} \, \bignorm{f_a}{1}.
	\end{equation}

	\medskip
	Finally, \textit{suppose that} $x$ is in $\Om_b'$.  Set
	$$
	\de_a(x) := d\big(x, \Om_a) \qquad\hbox{and}\qquad \si_b(x) := d(x, \partial \Om_b).
	$$
	Notice that if~$R$ is less than $\de_a(x)$, then $B_R(x)$ has empty intersection with~$\Om_a$, whence the average
	$\ds \frac{1}{\mod{B_R(x)}} \, \int_{B_R(x)} f_a \wrt\mu$ vanishes.  Thus, we can assume that $R\geq \de_a(x)$, and
	$$
	\sup_{R\geq \de_a(x)}\,\frac{1}{\bigmod{B_R(x)}} \, \int_{B_R(x)} f_a \wrt\mu
	\leq \frac{1}{\bigmod{B_{\de_a(x)}(x)}} \, \bignorm{f_a}{1}.
	$$
	Observe that
	$$
	\de_a(x)
	\geq \si_b(x)
	= d_b\big(\vp_b^{-1} (x), o_b\big)-1.
	$$
	Since $\vp_b$ is an isometry between $B_{\si_b(x)}^b\big(\vp_b^{-1} (x)\big)$ and $B_{\si_b(x)}(x)$, we can conclude that
	$$
	\bigmod{B_{\de_a(x)}(x)}
	\geq \mu_b\big(B_{\si_b(x)}^b\big(\vp_b^{-1}(x)\big)\big)
	\geq {\eta_{1,b} \, \, q^{b\si_b(x)}}
	$$
	(see \eqref{f: volume kappa} and recall that $q = \e^{m-1}$). 
	Altogether we have proved that
	$$
	\cM_\infty f_a (x)\leq
	\eta_{1,b}^{-1}\, q^{- b\si_b(x)} \bignorm{f_a}{1} \,
	\quant x \in \Om_b.
	$$
	Therefore
	$$
	\Bigmod{\Big\{x\in \Om_b': \cM_\infty f_a (x) > \la/3\Big\}}
	\leq \Bigmod{\Big\{x\in \Om_b': {q^{-b\si_b(x)} > \frac{\eta_{1,b}\,  \la}{3\,\norm{f_a}{1}}}\Big\}}.
	$$
	The latter measure is dominated by 
	$$
	\mu_b\Big(\Big\{v\in M_b: {q^{-bd_b(v, o_b)} > \frac{\eta_{1,b}\, \la}{3q^{b}\,\norm{f_a}{1}}}\Big\}\Big).
	$$
	Since the set above is {$B_{r(\la)}^b(o_b)$}, where
	$$
	r(\la) := {\frac{1}{b} \, \log_q\frac{3q^{b}\norm{f_a}{1}}{{\eta_{1,b}}\,\la}},
	$$
	we can conclude that
	\begin{equation} \label{f: fa III}
		\Bigmod{\Big\{x\in \Om_b': \cM_\infty f_a (x) > {\eta_{1,b}}\,\la/3\Big\}}
		\leq {\frac{3{\eta_{2,b}}\,q^{b}\norm{f_a}{1}}{{\eta_{1,b}}\,\la}}.
	\end{equation}

	By combining \eqref{f: fa I}, \eqref{f: fa II} and \eqref{f: fa III}, we obtain \eqref{f: estimate fa}, as required.

	A similar line of reasoning leads to the analogue of \eqref{f: estimate fa}, with $f_b$ in place of $f_a$.

	\medskip
	It remains to prove that
	\begin{equation} \label{f: estimate fK}
		\bigmod{\big\{\cM_\infty f_K> \la/3\big\}}
	\leq \frac{C}{\la} \, \bignorm{f_K}{1}
	\quant \la >0.
	\end{equation}
	Notice that $\cM_\infty f_K \leq \al^{-1}\, \norm{f_K}{1}$; the constant $\al$ is defined in \eqref{f: constant al}.

	Suppose that $d(x,K) \geq 3$.   Then
	$$
	\frac{1}{\bigmod{B_R(x)}}
	\int_{B_R(x)} f_K \wrt \mu
	\leq \frac{1}{\bigmod{B_{d(x,K)}(x)}}\, \bignorm{f_K}{1}.
	$$
	Observe that $B_{d(x,K)}(x)$ contains $\vp_a\big(B_{\si_a(x)-1}^a\big(\vp_a^{-1}(x)\big)\big)$ if $x$ belongs to~$\Om_a$, and it
	contains $\vp_b\big(B_{\si_b(x)-1}^b\big(\vp_b^{-1}(x)\big)$ if $x$ belongs to~$\Om_b$.  Hence
	$$
	\bigmod{B_{d(x,K)}(x)}
	\geq
	\begin{cases}
		{\eta_{1,a}}\, q^{a(\si_a(x)-1)}
		& {\mathrm{if} \, x\in \Om_a} \\
		{\eta_{1,b}}\, q^{b(\si_a(x)-1)}
		& {\mathrm{if} \, x\in \Om_b}.
	\end{cases}
	$$
	Thus,
	$$
	\begin{aligned}
		& \cM_\infty {f_K} (x) \\
		& \leq \big({\eta_{1,a}^{-1}}\,q^{-a(\si_a(x)-1)} \, \One_{\Om_a'}(x) + {\eta_{1,b}^{-1}}\,q^{-b(\si_b(x)-1)} 
			\, \One_{\Om_b'}(x) 
		+ \al^{-1} \, \One_{K'}(x) \big)
			\, \bignorm{f_K}{1},
	\end{aligned}
	$$
	so that $\bigmod{\big\{x \in M: \cM_\infty f_K (x)} > \la/3\big\}$ is dominated by the sum of the following three summands:
	$$
	\Bigmod{\Big\{x \in \Om_a': q^{-a\si_a(x)} > \frac{{\eta_{1,a}}\, \la}{9q^{a}\,\norm{f_K}{1}}\Big\}},
	\quad
	\Bigmod{\Big\{x \in \Om_b': q^{-b\si_a(x)} > \frac{{\eta_{1,b}}\, \la}{9q^{b}\,\norm{f_K}{1}}\Big\}}
	$$
	and
	$$
	\Bigmod{\Big\{x \in M: \One_{K'} (x) > \frac{\al\la}{9 \,\norm{f_K}{1}}\Big\}}.
	$$
	Set $\ga:= 9\norm{f_K}{1}/\al$.  A straightforward computation shows that the latter measure is equal to $\mu(K') \, \One_{(0,\ga)}(\la)$, 
	Therefore
	$$
	\Bigmod{\Big\{x \in M: \One_{K'} (x) > \frac{\al\la}{9 \,\norm{f_K}{1}}\Big\}}
	\leq \frac{9\, \mu(K')}{\al\, \la} \,  \norm{f_K}{1}.
	$$
	By arguing much as above one can show that the first two summands are dominated by 
	$9 \,{\eta_{2,a}}\,q^a\,\norm{f_K}{1}/({\eta_{1,a}}\la)$ and 
	$9 \,{\eta_{2,b}}\, q^b\,\norm{f_K}{1}/({\eta_{1,b}}\la)$, respectively.  
	By combining the estimates above we see that \eqref{f: estimate fK} holds, thereby completing the proof of \rmi.

	\medskip
	Next, we prove \rmii. Clearly, $\cN_\infty$ is bounded in $L^\infty(M)$. Hence suppose $p<\infty$.  For a large positive number $t$, consider the annuli
	$$
	E_t := \big\{x \in M_b: t\leq d(x,o_b) \leq t+1\big\}
	$$
	and
	$$
	F_t := \big\{x \in M_a: r_t \leq d(x,o_a) \leq r_t+1 \big\},
	$$
	where $r_t:= \big(2(b/a) - 1\big)\, t$.  We abuse the notation, and still denote by $E_t$ and $F_t$ the corresponding subsets in $M$.
	{By \eqref{f: volume kappa I}, given any  small preassigned positive number $\vep$, the estimates 
	\begin{equation} \label{f: annuli b I}
		(1-\vep)\, \eta_b\, \big(q^b-1\big) \, q^{bt} 
		\leq \bigmod{E_t} 
		\leq (1+\vep)\, \eta_b\, \big(q^b-1\big) \,q^{bt}
	\end{equation}
	and 
	\begin{equation} \label{f: annuli b II}
		(1-\vep)\, \eta_a\, \big(q^a-1\big) \, q^{at} 
		\leq \bigmod{F_t} 
		\leq (1+\vep)\, \eta_a\, \big(q^a-1\big) \,q^{at}.
	\end{equation}
	hold for every $t$ large enough.
	}
	We want to show that there exists a positive constant $c$, independent  of $t$, such that the level set $\{\cN \One_{E_t}>c\}$ contains $F_t$
	for every $t$ large enough.

	For each $x$ in $F_t$, consider the geodesic joining $x$ and~the point $z_a$ of $K$ closest to~$x$: clearly such geodesic ``agrees" with the geodesic
	joining $x$ and~$z_a$ in~$M_a$.  Hence $r_t-1\leq d(x,z_a) \leq r_t$.  Denote by $y$ the unique point on such geodesic
	at distance $(b/a) \, t+2$ from $x$.
	Of course $y$ depends on~$x$, but this has no consequence for what follows.  Since the gluing 
	is assumed to have the form \eqref{eq:model}, this geodesic then crosses $K$ and hits the boundary of the lower
	leave at a point~$z_b$. 
	
	Since both $z_a$ and~$z_b$ belong to~$K$, and, by construction, the diameter of $K$ is $\leq 3$, we have $d(z_a,z_b) \leq 3$.  Notice that
	$$
	d(y,z_a)
	= d(x,z_a)-d(x,y)
	= d(x,z_a)-(b/a)\, t -2,
	$$
	whence
	$$
	\big((b/a)-1\big)\, t - 3 \leq d(y,z_a) \leq \big((b/a)-1\big)\, t -2
	$$
	and
	$$
	\big((b/a)-1\big)\, t -3 \leq d(y,z_b) \leq \big((b/a)-1\big)\, t +1.
	$$
	Now, each point in $E_t$ is at distance at most $\be + t+1$ from $z_b$, where $\be$ denotes the length of the shortest path on the unit sphere
	of $M_b$ between $z_b$ and its antipodal point.  Note that, by homogeneity, $\be$ does not depend on the point $z_b$ considered.
	Altogether, we see that the distance from $y$ to each point of $E_t$ is at most
	$$
	\big((b/a)-1\big)\, t +1 + \be + t+1
	= (b/a)\, t + \be +2.
	$$
	
	Set $R_t :=(b/a) \, t +\be+ 2$.  The estimates above imply that~$B_{R_t}(y)$ contains both $x$ and $E_t$.
	Therefore, given $x$ in $F_t$, there exists~$y$ such that
	$$
	\cN \One_{E_t} (x)
	\geq \frac{1}{\bigmod{B_{R_t}(y)}} \, \int_{B_{R_t}(y)} \, \One_{E_t} \wrt \mu
	= \frac{\bigmod{E_t}}{\bigmod{B_{R_t}(y)}}.
	$$
	Thus, it remains to estimate $\bigmod{B_{R_t}(y)}$ from above.  Simple geometric considerations show that for $t$ large $B_{R_t}(y)$ is contained in
	$$
	\big[B_{R_t}^a(y) \setminus B_1^a(o_a)\big] \cup K \cup \big[B_{t+\be+2}^b (o_b)\setminus B_1^a(o_a)\big].
	$$
	The measure of this set is bounded above by
	$$
	\eta_{2,a} \, q^{aR_t} + \mod{K} + \eta_{2,b} \, q^{b(t+2\be+2)}.
	$$
	Since $aR_t= bt + a\be+2a$, we see that
	{
	$$
	\bigmod{B_{R_t}(y)} \leq \rho \, q^{bt},
	$$
	where $\rho$ is a constant, depending on $\be$, $a$, $b$ and the dimension $m$}.
Altogether, we can conclude that
$$
	\cN \One_{E_t} (x)
	\geq {\frac{\bigmod{E_t}}{\rho \, q^{bt}}
	\geq (1-\vep)\, \eta_b\, \big(q^b-1\big) \, \rho^{-1}}
	\quant x \in F_t. 
$$
Now, set $\ds c := (1-\vep)\, \eta_b\, \big(q^b-1\big) \, \rho^{-1}$.  Since $\cN \One_{E_t} \geq c$ on~$F_t$,
$$
	\frac{\bignormto{\cN \One_{E_t}}{p}{p}}{\bignormto{\One_{E_t}}{p}{p}}
	\geq c^p \, \frac{\bigmod{F_t}}{\bigmod{E_t}}.
$$
If $\cN$ were bounded on $\lp{M}$ for some $p$ in $[1,\infty)$,
then there would exist a constant~$C$, independent of $t$, such that $\bigmod{F_t} \leq C\, \bigmod{E_t}$ for all large $t$.
{Now \eqref{f: annuli b I} and \eqref{f: annuli b II} would imply that 
$$
(1-\vep)\, \eta_a\, \big(q^a-1\big) \,q^{ar_t}
\leq (1+\vep)\, \eta_b\, \big(q^b-1\big) \,q^{bt}
$$
for every $t$ large enough.  Clearly this is impossible, because $b>a$ by assumption.}

This concludes the proof of \rmii, and of the proposition.
\end{proof}

\section{Str\"omberg's counterexample} \label{s: Stromberg}

Suppose that $0<a<b$ and consider the positive function $\Psi$ such that
$$
\Psi(y)^2
:= \frac{1}{y^2}\, \Big[\frac{1}{b^2} + \frac{1/a^2-1/b^2}{y+1}\Big]. 
\quant y>0.
$$
It is straightforward to check that $\Psi$ is decreasing on $(0,\infty)$.
Consider the upper half plane $\Pi$, endowed with the Riemannian metric
\begin{equation} \label{f: ds2}
g
= \Psi(y)^2 \, \big(\!\wrt x^2 + \wrt y^2\big).
\end{equation}
Denote by $d$ the associated distance, by $B_R(z)$ the corresponding (open) ball with centre $z$ and radius $R$, and by $\mu$ the associated Riemannian
measure.
\smallskip
Note that
$$
1\leq (by)^2 \, \Psi(y)^2 \leq \Big(\frac{b}{a}\Big)^2
\quant y >0.
$$
Thus, $\Psi (y)^2 = \la^2(y) \, /(by)^2$, where $\la$ satisfies the bound $1< \la < b/a$.  Moreover, the Gaussian curvature of $(\Pi,g)$ is given by
\[
	K=-\frac{1}{2\Psi^{2}} \, \partial_y^2 \big[{\log}\big(\Psi^2) \big].  
\]
{A straightforward, albeit tedious, computation shows that $-b^2 \leq K \leq -a^2$.}

The purpose of this section is to give full details of the proof of the following result.

\begin{theorem} \label{t: Str}
	Suppose that $0<a<b$.  The following hold:
	\begin{enumerate}
		\item[\itemno1]
		if $b<2a$, then the operator $\cM_\infty$ is unbounded on $\lp{\Pi}$ for all $p<b/a$;
		\item[\itemno2]
		if $b>2a$, then the operator $\cM_\infty$ is unbounded on $\lp{\Pi}$ for all $p<\infty$.
	\end{enumerate}
\end{theorem}

\begin{remark}
	Part \rmi\ of this proposition is due to Str\"omberg (see \cite[p.~126]{Str}, where only some indications concerning its
	proof are given).  Here we provide full details of the proof.  We emphasise that such proof is reminiscent of the proof of the corresponding estimate
	in the case of trees with bounded geometry \cite[Proposition~3.3~\rmi]{LMSV}.
\end{remark}

For each $\kappa >0$, consider the positive function $\Psi_\kappa$ such that
$$
\Psi_\kappa(y)^2
:= \frac{1}{\kappa^2y^2}
\quant y>0.
$$
Denote by $d_\kappa$ the hyperbolic metric corresponding to the line element
$$
g_\kappa
= \Psi_\kappa(y)^2 \, \big(\!\wrt x^2 + \wrt y^2\big),
$$
by $B_R^\kappa(z)$ the corresponding ball with centre $z$ and radius $R$, and by $\mu_\kappa$ the associated Riemannian measure.
Note that
\begin{equation} \label{f: comparison a and b}
\Psi_b(y)
\leq \Psi(y)
\leq \Psi_a(y)
\quant y>0,
\end{equation}
that $\ds\Psi(y)^2 \sim (by)^{-2}$ as $y$ tends to $\infty$, and that  $\ds\Psi(y)^2 \sim (ay)^{-2}$ as $y$ tends to $0$.
Thus, $d$ is very much similar to $d_b$ in the region where $y$ is large and to $d_a$ when $y$ is small.  Clearly \eqref{f: comparison a and b} implies that
\begin{equation} \label{f: comparison a and b II}
	d_b\leq d \leq d_a
	\quad\hbox{and}\quad
	B_R^a(z) \subseteq B_R(z) \subseteq B_R^b(z)
\end{equation}
for every $z$ in $\Pi$ and every positive $R$.  We shall denote by $\ell(\ga)$ and $\ell_\kappa(\ga)$ the length of the curve $\ga$ with
respect to the distance $d$ and $d_\kappa$, respectively.

\begin{lemma} \label{l: L1 rests}
	The following hold:
	\begin{enumerate}
		\item[\itemno1]
			${\Psi_a}-{\Psi}$ is bounded on $(0,1)$; 
		\item[\itemno2]
			${\Psi}-{\Psi_b}$ belongs to $\lu{(1,\infty)}$.
	\end{enumerate}
\end{lemma}

\begin{proof}
	We prove \rmi.  We write
	$$
	\begin{aligned}
		{\Psi_a}-{\Psi}
		= \frac{\Psi_a^2-\Psi^2}{\Psi_a+\Psi}
		= \frac{1/a^2-1/b^2}{\ds (y+1) \, \Big[\frac{1}{a} + \Big(\frac{1}{b^2} + \frac{1/a^2-1/b^2}{y+1}\Big)^{1/2} \Big]}.
	\end{aligned}
	$$
	It is straightforward to check that the minimum on the interval $[0,1]$ of the denominator of the right hand side
	is equal to $2/a$, and it is attained at the origin.  Therefore
	\begin{equation} \label{f: sqrt bound}
	{\Psi_a}-{\Psi}
	\leq \frac{a}{2}\, \big(1/a^2-1/b^2\big),
	\end{equation}
	as required.

	Similarly, to prove \rmii\ we write
	\begin{equation}
	\begin{aligned}
 \label{f: sqrt boundII}
		{\Psi}-{\Psi_b}
		& = \frac{\Psi^2-\Psi_b^2}{{\Psi}+{\Psi_b}} \\
		& = \frac{1/a^2-1/b^2}{\ds y(y+1) \, \Big[\frac{1}{b} + \Big(\frac{1}{b^2} + \frac{1/a^2-1/b^2}{y+1}\Big)^{1/2} \Big]} \\
		& \leq \frac{b/2}{y(y+1)}\, \big(1/a^2-1/b^2\big),
	\end{aligned}
	\end{equation}
	and the right hand side belongs to $\lu{(1,\infty)}$, as required.
\end{proof}

It is convenient to set, for each $s>0$,
$$
\Pi_s
:= \{z\in \Pi: \Im z>s\}
\quad\hbox{and}\quad
L_s
:= \{z\in \Pi: \Im z=s\}.
$$
We also set $L_s^+ :=\{z\in L_s: \Re z \geq 0\}$.
For each $\te$ in $(-\pi/2,\pi/2)$ we denote by $\ga_{i,\te}$ the geodesic line through $i$ forming an angle $\te$ at $i$ with $L_1^+$.
In particular, $L_1$ is the tangent to $\ga_{i,0}$ at $i$.  
Denote by $A$ the region in $\Pi$ that lies below
$\ga_{i,0}$, and set
$$
F := \Pi \setminus (A\cup \Pi_1)
\quad\hbox{and}\quad
F_R := F\cap B_R(i).
$$
For every $z$ in $\Pi$ denote by $\ga_z$ the minimizing curve segment joining $i$ and~$z$.

\begin{lemma} \label{l: geodesic through i}
	The following hold:
	\begin{enumerate}	
        \item[\itemno1]
			$\ga_{i,0}$ is contained in the strip $\Pi\setminus\Pi_1$ and intersects the line~$L_1$ only at the point~$i$;
		\item[\itemno2]
			there exists a constant $\be_{a,b}$ such that
			$$
			d_a(i,z) - \be_{a,b}
			\leq d(i,z)
			\quant z \in A;
			$$
		\item[\itemno3]
			if $z\in F$, then the length minimizing segment $\ga_z$ (with respect to $d$) intersects $L_1$ at exactly one point $w_z$ besides $i$.
			Furthermore
			$$
			d_a(w_z,z)
			\leq d(w_z,z)+\be_{a,b}.
			$$
	\end{enumerate}
\end{lemma}

\begin{proof}
	Preliminarily we observe that a curve $\ga$ in $\Pi$ minimizing the length (with respect to $g$) between its extremal points must be strictly concave, 
	as long as the extremal points have different real parts. 

	Specifically, given $z_1$ and $z_2$ in the upper half plane such that $\Re z_1\neq \Re z_2$, the geodesic segment $\ga$ connecting $z_1$ to $z_2$ 
	lies above the geodesic $\si$ of $(\Pi,g_a)$ passing through $z_1$ and $z_2$ (i.e. the half circle which meets the horizontal $x$-axis orthogonally). 

	Indeed, suppose it is not the case.  By possibly considering a geodesic subinterval, we can suppose that $\ga$ lies entirely below $\si$.
    	Denote by $\ell$ the length of $\ga$ with respect to $g_a$ and choose a unit speed parametrisation for $\ga$ on $[0,\ell]$ with respect to $g_a$. 
	Define the new curve $\eta(t) := \pi\circ \ga(t)$, where $\pi$ denotes the nearest point projection in $(\Pi, g_a)$ onto $\si$. 
	Since $\pi$ is the projection on a convex set (the geodesic line $\si$) in the space $(\Pi, g_a)$ of negative curvature, 
	$\pi$ is $1$-Lipschitz with respect to $d_a$.  In particular, $\eta:[0,\ell]\to \Pi$ is a Lipschitz curve such that $\eta(0)=z_1$ and $\eta(\ell)=z_2$, and
    	\[
	\|\ga'(t)\|_{g_a}
	\ge \|\eta'(t)\|_{g_a}.
    		\]
	Now, note that 
	\[
	\Psi(y)^2
	=\frac{1}{a^2y^2}\left[\frac{a^2}{b^2}+ \frac{1-a^2/b^2}{y+1}\right]
	=:\frac{1}{a^2y^2}\Theta(y)^2
	\]
	with $\Theta$ non-increasing.  Since $\pi$ increases the imaginary part, for all $t$ in $[0,\ell]$ the following holds
	\[
	\|\ga'(t)\|_{g} 
	=   \Theta(\ga(t))\, \|\ga'(t)\|_{g_a}
	\ge \Theta(\eta(t))\, \|\eta'(t)\|_{g_a} 
	=   \|\eta'(t)\|_g.
	\]
	Accordingly, $\ga$ is (with respect to $g$) at least as long as $\eta$, whence $\ga=\eta$ by the uniqueness of geodesics in $(\Pi,g)$.

	Now, \rmi\ follows directly from this observation.

	\medskip
	Next we prove \rmii.
	Since $\Pi$ has negative curvature, any two geodesics starting at $i$ do not intersect again.  Thus
	$\ga_z$ is contained in $A$ with the exception of the point $i$.  Note that
	$$
	\begin{aligned}
		\ell(\ga_z)
		& = \int_{\ga_z} \Psi(y) \, \sqrt{(\!\wrt x)^2 + (\!\wrt y)^2}  \\
		& = \ell_a(\ga_z) + \int_{\ga_z} \big[{\Psi(y)} - {\Psi_a(y)}\big] \, \sqrt{(\!\wrt x)^2 + (\!\wrt y)^2}.
	\end{aligned}
	$$
	By Lemma~\ref{l: L1 rests} (and its proof),
	$$
	\sup_{0<y<1} \big[{\Psi_a} - {\Psi}\big]
	\leq \frac{a}{2} \, \big(1/a^2-1/b^2\big),
	$$
	so that the formula above implies that
	$$
	\ell_a(\ga_z)
	\leq \ell(\ga_z) + \frac{a}{2} \, \big(1/a^2-1/b^2\big)\, \ell_{\mathrm{E}} (\ga_z),
	$$
	where $\ell_{\mathrm{E}}(\ga_z)$ denotes the Euclidean length of $\ga_z$.
	Since $\ds \be := \!\!\sup_{z\in A} \ell_{\mathrm{E}} (\ga_z)$ is finite,
	$$
	\ell_a(\ga_z)
	\leq \ell(\ga_z) + \be_{a,b},
	$$
	where $\ds \be_{a,b} := \frac{\be a}{2} \, \big(1/a^2-1/b^2\big)$.   Now, $d_a(i,z) \leq \ell_a(\ga_z)$ and $d(i,z) = \ell(\ga_z)$, so that
	$d_a(i,z) \leq d(i,z) + \be_{a,b}$, which proves \rmii.

	\medskip
	Finally, we prove \rmiii.  By symmetry considerations, we see that it suffices to assume that $\Re z>0$.  Then the length minimizing segment~$\ga_{z}$
	joining~$i$ and~$z$ forms a positive angle with the half line $\{\Im \zeta =1, \Re \zeta>0\}$, so that~$\ga_z$ initially belongs to $\OV\Pi_1$.
	Then $\ga_z$ crosses the line $L_1$ and ends in the region $F$.

	Notice that since geodesics are concave,  $\ga_z$ intersects $L_1$ just at one point, which we denote by $w_z$.  
	
	Since the part of $\ga_z$ joining $i$ and $w_z$ is symmetric with respect to the line $\{\Re \zeta = (1/2)\, \Re w_z\}$, its tangent at $w_z$
	forms a negative angle with the half line $\{\Im \zeta = 1, \Re \zeta \geq \Re w_z\}$.  Since $d$ is invariant under horizontal translations of $\Pi$,
	the length minimizing segment joining $w_z$ and $z$ is contained in $A+\Re w_z$.  The required conclusion then follows from \rmii.
	
	This concludes the proof of the lemma.
\end{proof}

Suppose that $R$ is large.  We need to estimate the volume of $B_R(i)$ from above.  It is convenient to write
$$
B_R(i)
= \big[B_R(i) \cap A\big] \cup \big[B_R(i) \cap \OV{\Pi}_1\big] \cup F_R,
$$
and estimate the volume of the three subsets on the right hand side of the formula above separately.   This is done in the next lemma.

\begin{lemma} \label{l: geometry}
	There exists a positive constant $C$ such that the following hold for every $R >10$:
	\begin{enumerate}
		\item[\itemno1]
			$\bigmod{B_R(i)\cap A} \leq C \, \e^{aR}$;
		\item[\itemno2]
		$\bigmod{\big(B_{R}(i) \cap \OV\Pi_1\big)} \leq 4\, \big(\e^{bR/2}- 1\big)$;
		\item[\itemno3]
			if $b/2<a$, then $\mod{F_R} \leq C\, \e^{aR}$;
		\item[\itemno4]
			if $b/2<a$, then $\mod{B_R(i)} \leq C\, \e^{aR}$.
	\end{enumerate}
\end{lemma}

\begin{proof}
	First we prove \rmi.  By Lemma~\ref{l: geodesic through i}~\rmii, the region $B_R(i)\cap A$ is contained in $B_{R+\be_{a,b}}^a (i) \cap A$, and
	$$
	\begin{aligned}
		\bigmod{B_R(i)\cap A}
		& = \int_{B_R(i)\cap A} \Psi(y)^2 \wrt x\wrt y \\
		& \leq \int_{B_{R+\be_{a,b}}^a(i)\cap A} \Psi_a(y)^2 \wrt x\wrt y \\
		& = \mu_a\big(B_{R+\be_{a,b}}^a(i)\cap A\big),
	\end{aligned}
	$$
	so that
	$$
	\bigmod{B_R(i)\cap A}
	\leq \bigmod{B_{R+\be_{a,b}}^a(i)\cap A}
	\leq C\, \e^{aR},
	$$
	as required.

	\medskip
	Next, we prove \rmii.  The estimate \eqref{f: comparison a and b II} implies that 
	$$
	\big(B_{R}(i) \cap \OV\Pi_1\big)
	\subseteq \big(B_{R}^b(i)\cap \OV\Pi_1\big).
	$$
	It is well known that $B_{R}^b(i)$, which agrees with $B_{bR}^1(i)$, is the Euclidean disc with centre $i\cosh (bR)$ and radius $\sinh (bR)$.
	Therefore \eqref{f: comparison a and b} implies that
	$$
		\bigmod{\big(B_{R}(i) \cap \OV\Pi_1\big)}
		\leq \mu_a\big(B_{R}^b(i) \cap \OV\Pi_1\big)
		= \frac{1}{a^2}\, \mu_1\big(B_{bR}^1(i) \cap \OV\Pi_1\big).
	$$
	The boundary of the ball $B_{bR}^1(i)$ is the circle consisting of all points $x+iy$ such that
	$$
	x^2+\big(y-\cosh(bR)\big)^2 = \big(\sinh(bR)\big)^2.
	$$
	Elementary considerations show that this circle is contained in the circle with centre $(0,\e^{bR}/2)$ and radius $\e^{bR}/2$, whose equation is
	$$
	x^2+y(y-\e^{bR}) = 0.
	$$
	It is straightforward to check that the latter circle intersects the line $L_1$ at the points $\big(-\sqrt{\e^{bR}-1},1\big)$ and
	$\big(\sqrt{\e^{bR}-1},1\big)$
	and the vertical axis at the origin and at $(0,\e^{bR})$.  Thus,
	$$
		\begin{aligned}
			\bigmod{\big(B_{R}(i) \cap \OV\Pi_1\big)}
			& \leq \frac{2}{a^2}\,  \int_1^{\e^{bR}} \frac{\!\wrt y}{y^2} \int_0^{\sqrt{y(\e^{bR}-y)}} \wrt x \\
			& = \frac{2}{a^2}\,  \int_1^{\e^{bR}} \frac{\sqrt{\e^{bR}-y}}{y^{3/2}} \wrt y.
		\end{aligned}
	$$
	The change of variables $\e^{-bR}y = u$ transforms the last integral into
	$$
		\int_{\e^{-bR}}^1 \frac{\sqrt{1-u}}{u^{3/2}} \wrt u
		\leq \int_{\e^{-bR}}^1 \frac{1}{u^{3/2}} \wrt u
		= 2\, \big(\e^{bR/2}- 1\big).
	$$
	Therefore
	$$
		\bigmod{\big(B_{R}(i) \cap \OV\Pi_1\big)}
		\leq \frac{4}{a^2}\, \big(\e^{bR/2}- 1\big),
	$$
	as required.

	\medskip
	Now we prove \rmiii.
	The proof of \rmii\ shows that the projection of $B_{R}(i) \cap L_1$ on the horizontal axis is contained in   the interval
	$\big(-\sqrt{\e^{bR}-1},\sqrt{\e^{bR}-1}\big)$, which, in turn, is contained in
	$\big(\!-\e^{bR/2}, \e^{bR/2}\big)$.  Recall that
	$$
	d_1(x_1+iy,x_2+iy)
	= 2\, \arcsinh \frac{\bigmod{x_1-x_2}}{2y}
	\quant x_1,x_2\in \BR \quant y>0,
	$$
	and that $d_b=(1/b)\, d_1$.  Thus,
	$$
	d_b\big(i,i+\e^{bR/2})
	= \frac{1}{b} \, d_1\big(i,i+\e^{bR/2})
	= \frac{2}{b} \, \log\Big(\frac{\e^{bR/2}}{2} + \sqrt{\frac{\e^{bR}}{4}+1}\Big).
	$$
	Since
	$$
	0<
	\frac{2}{b} \, \log\Big(\frac{1 + \sqrt{1+4\e^{-bR}}}{2}\Big)
	<\frac{2}{b},
	$$
	we can conclude that $R<d_b\big(i,i+\e^{bR/2})<R+(2/b)$.  It is straightforward to see that $d_b\big(i,i+\e^{bR/2}) = R + O(\e^{-bR})$,
	as~$R$ tends to infinity.
	
	Furthermore, observe that for each pair of points $\zeta_1$ and $\zeta_2$ on $L_1$, we have that $d_b(\zeta_1,\zeta_2) = 1$ if and only if
	$(2/b) \, \arcsinh \big(\bigmod{\Re\zeta_1-\Re\zeta_2}/2\big) = 1$, viz. if and only if
	\begin{equation} \label{f: distance sigma}
		\bigmod{\Re\zeta_1-\Re\zeta_2} = 2 \sinh(b/2).
	\end{equation}
	Let $\si_0,\si_1,\ldots$ be the points on $L_1$ such that
	$$
	\si_0 = i, \quad \Re \si_j < \Re \si_{j+1}
	\quad\hbox{and}\quad d_b(\si_j,\si_{j+1}) = 1
	\quant j\in \BN,
	$$
	and choose $N$ such that $\e^{bR/2} \leq \Re \si_N \leq \e^{bR/2}+1$.  Notice that
	$$
	d_a(\si_j,\si_{j+1}) = b/a
	\qquad\hbox{and}\qquad
	1<d(\si_j,\si_{j+1}) < b/a.
	$$

	For every $z$ in $F_R$ consider the intersection $w_z$ (see Lemma~\ref{l: geodesic through i}~\rmiii\ for the notation) between the length minimizing
	curve $\ga_z$ and $L_1$, and set $r_z := d(i,w_z)$.  Then $d(w_z,z) < R - r_z$.  Denote by $\si_{j_z}$ (one of)
	the point(s) amongst $\si_0,\ldots,\si_N$ closest to $w_z$ with respect to $d_b$.   Since $d_b(w_z,\si_{j_z}) \leq 1$,
	and $d_a = (b/a) \, d_b$, we find that $d_a(w_z,\si_{j_z}) \leq b/a$.  The triangle inequality now yields
	$$
	d_a(z,\si_{j_z})
	\leq d_a(z,w_z) + d_a(w_z,\si_{j_z})
	\leq R-r_z+\be_{a,b}+(b/a).
	$$
	The last inequality follows from Lemma~\ref{l: geodesic through i}~\rmiii.  Now, \eqref{f: comparison a and b} implies that
	$$
	r_z = d(i,w_z) \geq d_b(i,w_z),
	$$
	so that
	$$
	d_a(z,\si_{j_z})
	\leq R-d_b(i,w_z)+\be_{a,b}+b/a.
	$$
	By the triangle inequality
	$$
	d_b(i,w_z)\geq d_b(i,\si_{j_z})-d_b(\si_{j_z},w_z)\geq d_b(i,\si_{j_z})-1,
	$$
	so that
	$$
	d_a(z,\si_{j_z})
	\leq R-d_b(i,\si_{j_z})+1+\be_{a,b}+b/a.
	$$
	Set $C_{a,b} := 1+\be_{a,b}+(b/a)$.	The estimate above implies that
	\begin{equation} \label{f: first inclusion}
	F_R
		\subseteq \ds\bigcup_{j=0}^N B_{R-d_b(i,\si_{j})+C_{a,b}}^a(\si_j).
	\end{equation}
	By \eqref{f: distance sigma} there are approximately
	$$
	\frac{\e^{b(k+1)/2}-\e^{bk/2}}{2\sinh(b/2)} = \frac{\e^{bk/2}}{1+\e^{-b/2}}
	$$
	points of the finite sequence $\si_1,\ldots,\si_N$ between $i+\e^{bk/2}$ and $i+\e^{b(k+1)/2}$.  If $\si_j$ is one such point, then
	$$
	d(i,\si_j)
	\geq k.
	$$
	Thus, we can re-write the right hand side of \eqref{f: first inclusion} as follows
	$$
	\bigcup_{k=1}^{\lfloor R\rfloor+1} \bigcup_{j\in \al(k)}  B_{R-d_b(i,\si_{j})+C_{a,b}}^a(\si_j),
	$$
	where $\al(k) := \big\{j: \e^{bk/2}\leq \Re(\si_j)\leq \e^{b(k+1)/2}\big\}$.  We denote by $\si_k^*$ the point amongst $\{\si_j: j\in \al(k)\}$
	with smallest distance from $i$.  Then
	\begin{equation} \label{f: last step}
	\begin{aligned}
		\mod{F_R} \leq C\, \mu_a(F)
		& \leq C \, \sum_{k=1}^{\lfloor R\rfloor+1}\, \mu_a\big(B_{R-d_b(i,\si_k^*)+C_{a,b}}^a(\si_k^*)\big)\,  \sharp \al(k)\\
		& \leq C \, \sum_{k=1}^{\lfloor R\rfloor+1}\, \e^{a(R-k)}\, \e^{kb/2}, \\
	\end{aligned}
	\end{equation}
	which is dominated by $C\, \e^{aR}$, because the series $\ds \sum_{k=1}^\infty   \, \e^{k(b/2-a)}$ is convergent (for $b<a/2$ by assumption).
	This concludes the proof of \rmiii.
	
\medskip
	Finally, \rmiv\ follows directly from \rmi-\rmiii, and the assumption $b/2<a$.
\end{proof}

We are now in position to prove Theorem~\ref{t: Str}.

\begin{proof}
	First we prove \rmi.  Suppose that $t$ is a large positive number, and denote by $f_t$ the characteristic function of the ball $B_1(it)$.  Define
	$$
	E_t
	:= (-t,t)\times (1,2). 
	$$
	Observe that for each $z$ in $E_t$, the following obvious lower bound holds
	$$
	\cM_\infty f_t(z)
	\geq \frac{\bigmod{B_1(it)}}{\bigmod{B_{d(z,it)+1}(z)}}.
	$$
	Notice that the local doubling condition implies that there exists a constant~$C$, independent of $t$, such that
	$$
	\bigmod{B_{d(z,it)+1}(z)} \leq C\, \bigmod{B_{d(z,it)}(z)}.
	$$
Clearly $d(z,it)$ is smaller that the length of the path $\ga_{z,it}$ consisting of the vertical segment
	joining $z$ and $\Re z+it$ and the horizontal curve joining $\Re z+it$ and $it$.  Thus,
	$$
	\begin{aligned}
\ell(\ga_{z,it})
		& = \int_1^t {\Psi(y)} \wrt y + {\Psi(t)} \,\, \Bigmod{\int_{\Re z}^0 \wrt x} \\
		& = \int_1^t \big[{\Psi(y)}-{\Psi_b(y)}\big] \wrt y + \ell_b([z,z+it]) + {\Psi(t)} \,\, \mod{\Re z},
	\end{aligned}
	$$
	where $\ell_b([z,z+it])$ denotes the length of the segment $[z,z+it]$ in the hyperbolic metric $d_b$.  By Lemma~\ref{l: L1 rests} the
	integral on the right hand side is bounded by a constant independent of $t$ (and $z$).   Clearly
	$$
	\ell_b([z,z+it])
	= \frac{1}{b} \, \int_1^t \frac{1}{y} \wrt y
	= \log t^{1/b},
	$$
	and, for {{$t\geq 1$}},
	$$
	{\Psi(t)}
	\leq \frac{1}{\sqrt 2\, t}\,  \big[1/a^2+1/b^2\big]^{1/2}.
	$$
	It is convenient to set $R_t := \log t^{1/b}$.   Altogether, we see that
	$$
		\begin{aligned}
			\ell(\ga_{z,it})
			& \leq R_t + \bignorm{{\Psi}-{\Psi_b}}{\lu{(1,\infty)}} + \frac{1}{\sqrt 2}\,  \big[1/a^2+1/b^2\big]^{1/2} \\
			& \leq R_t + \al_{a,b}
			\quant z\in E_t.
		\end{aligned}
	$$
	This and the local doubling condition then imply the estimate
	\begin{equation} \label{f: balls II}
		\bigmod{B_{d(z,it)}(z)}
		\leq C\, \bigmod{B_{R_t}(z)}
		\quant z \in E_t.
	\end{equation}
	{Now, since $d$ is invariant with respect to horizontal translations (the factor $\Psi$ depends only on the vertical variable),
	$\bigmod{B_{R_t}(z)} = \bigmod{B_{R_t}(i\Im z)}$.  {By}} the triangle inequality  $B_{R_t}(i\Im z) \subset B_{R_t+1}(i)$,
	so that
	$$
	\bigmod{B_{R_t}(z)}
	\leq \bigmod{B_{R_t+1}(i)}
	\leq C \, \bigmod{B_{R_t}(i)}.
	$$
	Therefore, by combining the estimates above, we obtain that there exists a positive constant such that
	\begin{equation} \label{f: lower max II}
	\cM_\infty f_t(z)
		\geq \frac{C}{\bigmod{B_{R_t}(i)}}
		\quant z \in E_t.
	\end{equation}
	This implies that the level set $\big\{z \in \Pi: \cM_\infty f_t(z) > (C/2)/\bigmod{B_{R_t}(i)} \big\}$ contains $E_t$.
	Thus, if $\cM_\infty$ is of weak type $p$, then there must exist a constant~$C$ such that
	\begin{equation} \label{f: final Str}
		\bigmod{E_t}
		\leq C\, \bignormto{f_t}{p}{p}\,\,  \bigmod{B_{R_t}(i)}^p
	\end{equation}
	for every $t$ large enough.  Clearly
	\begin{equation} \label{f: Et}
		\bigmod{E_t}
		= \int_{-t}^t\int_1^2 \Psi(y)^2 \, \wrt x \wrt y
		\geq \int_{-t}^t\int_1^2  \frac{1}{b^2y^2} \wrt x \wrt y
		= t/b^2,
	\end{equation}
	and{ {, by \eqref{f: comparison a and b} and \eqref{f: comparison a and b II},}}
	\begin{equation} \label{f: ft}
		\sup_{t>0} \bignormto{f_t}{p}{p}
		= \sup_{t>0} \, \bigmod{B_{1}(it)} 
		\leq \sup_{t>0}\, \mu^a(B^b_1(it))
        	< \infty.
	\end{equation}
	By Lemma~\ref{l: geometry}, there exists a positive constant $C$ such that
	$$
		\bigmod{B_{R_t}(i)}
		\leq C \, \e^{aR_t}
	$$	
	for all large $t$.  By combining \eqref{f: final Str}, \eqref{f: Et} and \eqref{f: ft}, we see that
	$t \leq C \, t^{pa/b}$ for all large~$t$.  Clearly this fails whenever~$p<b/a$.

	\medskip
	Next we prove \rmii.  The proof is reminiscent of the proof of \cite[Proposition~3.3~\rmii]{LMSV}.  Suppose that $t$ is a large positive number,
	and consider the set
	$$
	E_t
	= (-t,t)\times (1,2), 
	$$
	that already appears in the proof of \rmi.  Set $\tau := b/a$, and observe that $\tau>2$ under our assumptions.
	Notice that
	$$
	d(i,it)
	= \int_1^{t} {\Psi(y)} \wrt y
	\geq \frac{1}{b} \, \int_1^{t} {\dtt y}
	= \log t^{1/b},
	$$
	because ${\Psi} \geq {\Psi_b}$.  Similarly,
	\begin{equation}
    \label{d(i,it)}
	\begin{aligned}
	d(i,it)
		& \leq \frac{1}{b} \, \int_1^{t} {\dtt y} + \bignorm{{\Psi} - {\Psi_b}}{\lu{[1,\infty}}\\
		& \leq \log t^{1/b} + \frac{b}{2} \, \big(1/a^2-1/b^2\big);
	\end{aligned}
	\end{equation}
	the last inequality follows from \eqref{f: sqrt boundII}.

	Consider also the set
	$$
	F_t
	:= (-t,t)\times \big(t^{2/\tau-1}, 2t^{2/\tau-1}\big).
	$$
	Notice that
	\begin{equation} \label{f: upper 2tau}
	d(i,it^{2/\tau-1})
	= \int_{t^{2/\tau-1}}^1 {\Psi(y)} \wrt y
	\leq \frac{1}{a} \, \int_{t^{2/\tau-1}}^1 {\dtt y}
	= \log t^{(1-2/\tau)/a},
	\end{equation}
	because ${\Psi} \leq {\Psi_a}$.  Similarly,
		\begin{equation} \label{f: lower 2tau}
	\begin{aligned}
	d(i,it^{2/\tau-1})
		& \geq \frac{1}{a} \, \int_{t^{2/\tau-1}}^1 {\dtt y} - \sup_{0<y\leq 1} \big[{\Psi_a (y)} - {\Psi(y)}\big] \\
		& \geq \log t^{(1-2/\tau)/a} - \frac{a}{2} \, \big(1/a^2-1/b^2\big);
	\end{aligned}
		\end{equation}
	the last inequality follows from \eqref{f: sqrt bound}.

	The required conclusion will be a direct consequence of the following four claims:
	\begin{enumerate}
		\item[\itemno1]
			there exists a constant $\be$, independent of $t$, such that $B_{r_t}(z)$ contains~$E_t$ for every $z$ in~$F_t$:
			here $r_t := \log t^{1/a} + \be$;
		\item[\itemno2]
			$\mod{F_t}/\bigmod{E_t}$ is unbounded as $t$ tends to infinity;
		\item[\itemno3]
			there exists a constant $C$ such that
			$$
			\bigmod{B_{r_t}(z)} \leq C\, t
			\quant z \in F_t
			$$
			for all large $t$;
		\item[\itemno4]
			there exists a positive constant $\la_0$ such that $\mod{E_t}/\bigmod{B_{r_t}(z)} > \la_0$ for every $z$ in $F_t$.
			Consequently $F_t$ is contained in the set where $\cM_\infty \One_{E_t} > \la_0$.
	\end{enumerate}
	
	\noindent
	Given the claims, the required conclusion follows easily.  We proceed by \textit{reductio ad absurdum}.  If $\cM_\infty$ were of weak type $(p,p)$
	on $\lp{M}$ for some finite $p$, there would exist a constant $C$ such that
	$$
	\bigmod{\{x \in M: \cM_\infty f (x)> \la \}}
	\leq \frac{C}{\la^p}\,\bignormto{f}{p}{p}
	\quant \la> 0 \quant f \in \lp{M}.
	$$
	Choose $f = \One_{E_t}$ and $\la = \la_0$ (see \rmiv\ above for the notation).  By \rmi, the left hand side of the inequality above is
	bounded from below by $\mod{F_t}$.  Therefore we would have
	$$
	\mod{F_t}
	\leq \frac{C}{\la_0^p} \, \mod{E_t}
	$$
	for every large $t$.  However, the latter inequality fails because of {\rmii\ } above.

	\smallskip
	Now we prove the claims \rmi-\rmiv\ above.
	A straightforward computation shows that $\mod{r_t - \tau \, d(i,it)}$ and $\mod{r_t - 2d(i,it)-d(i,it^{2/\tau-1})}$ are bounded
	by a constant independent of $t$.
	
	We now prove that there exists a constant $\be$ such that for every $t$ large enough every point in $E_t$ is reachable from every point
	in $F_t$ with a path of length at most {$\log t^{1/a} +\be$}.   Suppose that $z$ is in $F_t$ and $w$ is in $E_t$.
	Clearly $d(z,w)$ is smaller that the length of the path $\ga_{z,w}$ consisting of the vertical segment
	joining $z$ and $\Re z+it$, the horizontal curve joining $\Re z+it$ and $\Re w + it$ and the vertical segment joining $\Re w + it$
	and $w$.  Thus,
	$$
	\begin{aligned}
		\ell(\ga_{z,w})
		& = d(z, \Re z+it) + {\Psi(t)} \,\, \Bigmod{\int_{\Re z}^{\Re w} \wrt x} + d(w, \Re w+it)
	\end{aligned}
	$$
	Note that $\mod{\Re z- \Re w} < 2t$, because both $\Re z$ and $\Re w$ belong to the interval $(-t,t)$.
	This and the right hand inequality in \eqref{f: comparison a and b} imply that the second summand on the right hand side of the formula
	above is dominated by $2/a$.

	{Using \eqref{d(i,it)} and \eqref{f: upper 2tau}}, it is straightforward to check that 
	$$
	d(z, \Re z+it) 	
	\leq \log t^{1/a-1/b} +  C 
	\quad\hbox{and}\quad
	d(w, \Re w+it)
	\leq \log t^{1/b} +  C 
    ,
	$$
	{with $C:= \big(1/a^2-1/b^2\big)\, b/2 $.} Therefore
	$$
	\ell(\ga_{z,w})
	\leq \log t^{1/a-1/b} + \log t^{1/b} + b \, \big(1/a^2-1/b^2\big)
	=  \log t^{1/a} + \beta,
	$$
	as required.
	As a consequence, $E_t$ is contained in $B_{r_t}(z)$ for every $z$ in $E_t$, and claim \rmi\ is proved.

	\smallskip
	Next we prove claim \rmii.  By arguing much as in $\eqref{f: Et}$, we see that $\bigmod{E_t} \leq t/a^2$.  Furthermore
	\begin{equation} \label{f: Ft}
		\begin{aligned}
		\bigmod{F_t}
			& = \int_{-t}^{t} \int_{t^{2/\tau-1}}^{2t^{2/\tau-1}} \Psi(y)^2 \, \wrt x \wrt y \\
			& \geq \int_{-t}^{t} \int_{t^{2/\tau-1}}^{2t^{2/\tau-1}} \frac{1}{b^2y^2} \wrt x \wrt y \\
			& = \frac{1}{b^2}\, t^{2-2/\tau}.
		\end{aligned}
	\end{equation}
	By combining the estimates above we see that
	$$
	\frac{\mod{F_t}}{\bigmod{E_t}}
	\geq C\, t^{1-2/\tau},
	$$
	which is unbounded as $t$ tends to infinity (because $\tau>2$), thereby proving claim \rmii.

	\smallskip
	Now we prove \rmiii.  It is convenient to set $p_t := it^{2/\tau-1}$.
	Since the metric is invariant under horizontal translations, it suffices to assume that $z$ belongs
	to the vertical axis.  Notice that the distance between any two points in $F_t$ that lie on the vertical axis is at most
	$$
	d(p_t, 2p_t)
	\leq \int_{t^{2/\tau-1}}^{2t^{2/\tau-1}}\, {\Psi(s)} \wrt s
	\leq \frac{1}{a} \, \int_{t^{2/\tau-1}}^{2t^{2/\tau-1}}\, {\dtt s}
	\leq \frac{\log 2}{a}.
	$$
	Therefore for each $z$ in $F_t$ lying on the vertical axis
	$$
	B_{r_t}(z) \subseteq B_{r_t+(\log2)/a}(p_t),
	$$
	so that
	$$
	\bigmod{B_{r_t}(z)}
	\leq \bigmod{B_{r_t+(\log2)/a}(p_t)}
	\leq C \, \bigmod{B_{r_t}(p_t)},
	$$
	where $C$ does not depend on $t$.  Hence it suffices to show that
	\begin{equation} \label{f: suffices}
	\bigmod{B_{r_t}(p_t)}
	\leq C\, t
	\end{equation}
	for all $t$ large enough.  Clearly $B_{r_t}(p_t)$ is symmetric with respect to the vertical axis.  Thus, it suffices to estimate
	the area of the right half of it.
	
	For any $z$ in $\Pi$, denote by $\Ga_z$ the unique geodesic joining $p_t$ and $z$.

        As observed at the beginning of the proof of Lemma~\ref{l: geodesic through i},
	the geodesic $\ga_{i,0}$ through~$i$ with horizontal tangent is concave.
	Thus, it must cross the negative real half line at a point $-x_0<0$ and there exists a point $-x_t+p_t$ on the geodesic, with $0<x_t<x_0$,
	at height $t^{2/\tau-1}$.
	Since the metric $d$ is invariant by horizontal translations, the geodesic $\ga_{i+{x_t},0}$, which is equal to $x_t+\ga_{i,0}$, 
	is tangent to $L_1$ at $i+{x_t}:=z_t$ and intersects the imaginary axis at $p_t$.Thus it is the minimising geodesic $\Ga_{z_t}$ joining $p_t$ 
	to $z_t$.  Note that since $0<x_t<x_0$, 
 	$z_t$ is on the left of $i+\mod{x_0}$
	and therefore
	\begin{equation} \label{f: x0}
	\eta := \sup_{t\geq 1} \, d(i,z_t)
	\leq d(i,i+\mod{x_0}).
	\end{equation}
	Denote by $\wt\Ga$ the geodesic with horizontal tangent at $p_t$.  Now, set
	$$
	A_t
	:= \{z \in \Pi: \Re z > 0\,\,\hbox{and $z$ lies below $\Ga_{z_t}$}\},
	$$
	$$
	\wt A_t
	:= \{z \in \Pi\setminus\Pi_1: 0 < \Re z < \Re z_t \,\,\hbox{and $z$ lies above $\Ga_{z_t}$}\},
	$$
	and
	$$
	G_t
	:= \{z \in \Pi\setminus \big(\OV{\Pi_1}\cup A_t\big): \Re z > \Re z_t \}.
	$$

	For each $z$ in $A_t$ the geodesic $\Ga_z$ is contained in $\Pi\setminus \Pi_1$.  Therefore, by arguing much as in the proof of
	Lemma~\ref{l: geodesic through i}~\rmii, we see that $d_a(p_t,z) \leq d(p_t,z) + \be_{a,b}$.  Hence $B_{r_t}(p_t) \cap A_t$
	is contained in $B_{r_t+\be_{a,b}}^a(p_t) \cap A_t$, so that
	$$
	\bigmod{B_{r_t}(p_t) \cap A_t}
	\leq \bigmod{B_{r_t+\be_{a,b}}^a(p_t) \cap A_t}
	\leq C\, \bigmod{B_{r_t}^a(p_t)}
	\leq C\, \e^{ar_t}
	= C\, t.
	$$
	Similarly, if $z$ is in $\wt A_t$, then the geodesic $\Ga_z$ is contained in the region below the geodesic $\ga_{i+\Re z, 0}$ and 
	therefore in $\Pi\setminus \Pi_1$ so, again, $d_a(p_t,z) \leq d(p_t,z) + \be_{a,b}$ and
	$$
	\bigmod{B_{r_t}(p_t) \cap \wt A_t}
	\leq C\, t.
	$$

	Next we estimate $\mod{B_{r_t}(p_t) \cap \Pi_1}$.  For each $z$ in $B_{r_t}(p_t) \cap \Pi_1$,
	denote by $w_z$ the intersection between $\Ga_z$ and $L_1$.  Note that
	$$
	r_t	>
    d(p_t,z)
	= d(p_t,w_z) + d(w_z,z),
	$$
	because $w_z$ belongs to the geodesic joining $p_t$ and $z$.
	Notice also that $d(i,w_z) \leq d(i,z_t)\leq \eta$ (the constant $\eta$ is defined in \eqref{f: x0}).
	By the triangle inequality	
	\begin{equation} \label{f: diz}
	d(i,z)
	\leq d(i,w_z) + d(w_z,z)
	\leq \eta + r_t - d(p_t,w_z)
	\leq \eta + r_t - d(p_t,i);
	\end{equation}
the last inequality follows from the trivial fact that $d(p_t,i) \leq d(p_t,w_z)$.
	By Lemma~\ref{l: geodesic through i}, we see that  $d(p_t,i)\geq \log t^{(1-2/\tau)/a} - \be_{a,b}$, which, combined with \eqref{f: diz}, yields
	$$
	d(i,z)
	\leq \eta + r_t - \log t^{(1-2/\tau)/a} + \be_{a,b}
	= \eta + \log t^{2/(a\tau)} + \be_{a,b}.
	$$
	This shows that
	$$
	B_{r_t}(p_t) \cap \Pi_1
	\subseteq B_{s_t+\eta+\be_{a,b}}(i) \cap \Pi_1,
	$$
	where $s_t := \log t^{2/(a\tau)}$.  Therefore
	$$
	\bigmod{B_{r_t}(p_t) \cap \Pi_1}
	\leq C\, \bigmod{B_{s_t}(i) \cap \Pi_1}
	\leq C\, \e^{bs_t/2}
	= C\, t:
	$$
	the last inequality follows from Lemma~\ref{l: geometry}~\rmii.

	In order to conclude the proof of \eqref{f: suffices}, it suffices to show that
	$$
	\bigmod{B_{r_t}(p_t) \cap G_t}
	\leq C\, t
	$$
	for all $t$ large enough.  By arguing much as above, we see that
	$$
	\bigmod{B_{r_t}(p_t) \cap G_t}
	\leq C\, \bigmod{B_{s_t}(i) \cap G_t}.
	$$
	By proceeding almost \textit{verbatim} as in the proof Lemma~\ref{l: geometry}~\rmiii\ (with $s_t$ in place of $R$), we arrive at the estimate
	$$
	\bigmod{B_{s_t}(i) \cap G_t}
	\leq C \, \sum_{k=1}^{\lfloor s_t\rfloor+1}\, \e^{a(s_t-k)}\, \e^{kb/2},
	$$
	which is the analogue of \eqref{f: last step}.  Clearly the right hand side above can be re-written as
	$$
	C \, \e^{as_t}\,\sum_{k=1}^{\lfloor s_t\rfloor+1}\, \e^{k(b/2-a)}
	\leq C \, \e^{(a+b/2-a)s_t}
	= C\, t;
	$$
	we have used the fact that $b>2a$ in the inequality above.

	This concludes the proof of claim \rmiii.

	\smallskip
	Finally, claim \rmiv\ is a straightforward consequence of the upper estimate of $\mod{B_{r_t}(z)}$ in claim \rmiii\ and the lower estimate \eqref{f: Et}
	of $\mod{E_t}$.
\end{proof}

\end{document}